\documentclass{amsart}

\theoremstyle{definition}

\theoremstyle{remark}

\numberwithin{equation}{section}
\usepackage{graphicx}
\usepackage{latexsym}
\usepackage{subfigure}
\usepackage{amsmath}
\usepackage{amssymb}
\usepackage{amsfonts}
\usepackage{verbatim}
\usepackage{mathrsfs}
\usepackage[latin1]{inputenc}
\usepackage{color}
\usepackage[colorlinks,citecolor=blue,urlcolor=blue]{hyperref}
\usepackage{caption}

\usepackage{datetime}

\newtheorem{tm}{Theorem}[section]
\newtheorem{rk}{Remark}[section]

\newtheorem{prop}{Proposition}[section]
\newtheorem{lm}{Lemma}[section]
\newtheorem{cor}{Corollary}[section]
\newtheorem{ex}{Example}[section]
\usepackage{enumitem} 
\newcommand{\E}{\mathbb E}

\newcommand{\bi}{\mathbf i}

\newcommand{\<}{\langle}
\renewcommand{\>}{\rangle}
\numberwithin{figure}{section}
\usepackage{geometry}
\newcommand{\zhou}[1]{{\color{red} [zhou: #1]}}
\newcommand{\cui}[1]{{\color{blue} [cui: #1]}}

\begin{document}

\title[Stochastic Wasserstein Hamiltonian Flows]
{Stochastic Wasserstein Hamiltonian Flows}

\author{Jianbo Cui}
\address{Department of Applied Mathematics, The Hong Kong Polytechnic University, Hung Hom, Kowloon, Hong Kong}
\curraddr{}
\email{jianbo.cui@polyu.edu.hk}
\thanks{The research is partially supported by Georgia Tech Mathematics Application Portal (GT-MAP) and by research grants NSF  DMS-1830225, and ONR N00014-21-1-2891. The research of the first author is
partially supported by start-up funds from Hong Kong Polytechnic University and the CAS AMSS.PolyU Joint Laboratory of Applied Mathematics.}
\author{Shu Liu}
\address{School of Mathematics, Georgia Tech, Atlanta, GA 30332, USA}
\curraddr{}
\email{sliu459@gatech.edu}
\thanks{}
\author{Haomin Zhou}
\address{School of Mathematics, Georgia Tech, Atlanta, GA 30332, USA}
\curraddr{}
\email{hmzhou@math.gatech.edu}
\thanks{}




\dedicatory{}

\begin{abstract}
In this paper, we study the stochastic Hamiltonian flow in Wasserstein manifold, the probability density space equipped with $L^2$-Wasserstein metric tensor, via the Wong--Zakai approximation. We begin our investigation by showing that the stochastic Euler-Lagrange equation, regardless it is deduced from either variational principle or particle dynamics, can be interpreted as the stochastic kinetic Hamiltonian flows in Wasserstein manifold. We further propose a novel variational formulation to derive more general stochastic Wassersetin Hamiltonian flows, and demonstrate that this new formulation is applicable to various systems including the stochastic Schr\"odinger equation, Schr\"odinger equation with random dispersion, and Schr\"odinger bridge problem with common noise.   
\end{abstract}


\maketitle

\section{Introduction}

The density space equipped with $L^2$-Wasserstein metric forms an infinite dimensional Riemannain manifold, often called Wasserstein manifold or density manifold in literature (see e.g. \cite{MR924776}). It plays an important role in optimal transport theory \cite{Vil09}.  Many well-known equations, such as Schr\"odinger equation, Schr\"odinger bridge problem and Vlasov equation, can be written as Hamiltonian systems on the density manifold. In this sense, they can be considered as members of the so-called Wasserstein Hamiltonian flows
(\cite{Vil09,MR2361303,MR2808856,CP17,CLZ19,CLZ20,CLZ21}). The study of Wasserstein Hamiltonian flow can be traced back to Nelson's mechanics (\cite{Nelson19661079,MR0343816,MR783254,MR870196}). 
Recently, it is shown in \cite{CLZ20} that the kinetic Hamiltonian flows in density space are probability transition equations of classical Hamiltonian ordinary differential equations (ODEs). In other words, this reveals that the density of a Hamiltonian flow in sample space is a Hamiltonian flow on density manifold. 

In the existing works on Wasserstein Hamiltonian flows, random perturbations to the Lagrangian functional are not considered. Consequently, the theory is neither directly applicable to the Wasserstein Hamiltonian flows subjected to random perturbations, nor to the systems whose parameters are not given deterministically. The main goal of this article is developing a theory to cover these scenarios in which the stochasticity is presented. More precisely,  we mainly focus on the stochastic perturbation of the Wasserstein Hamiltonian flow,
\begin{align*}
d\rho_t &=\frac {\delta} {\delta S_t} \mathcal H_0(\rho_t,S_t)dt,\\
dS_t&=-\frac {\delta} {\delta \rho_t} \mathcal H_0(\rho_t,S_t)dt,
\end{align*}
with a Hamiltonian $\mathcal H_0$ on the density manifold and $\frac {\delta}{\delta S},\frac {\delta}{\delta \rho} $ being the variational derivatives,  
which is proposed by only imposing randomness on the initial position of the  phase space \cite{CLZ20}.
This is different from the Hamiltonian flows considered in \cite{MR2361303}, where
the authors consider and construct the solutions of the ODEs in the measure space of even dimensional phase variables equipped with the Wasserstein metric.

To study the stochastic variational principle on density manifold, we may  confront several challenges. First and the foremost, the Wasserstein Hamiltonian flow studied in \cite{CLZ20} is induced based on the principle that the density of a Hamiltonian flow in sample space is a Wasserstein Hamiltonian flow in density manifold. This principle may no long hold if the Hamiltonian flow in sample space is perturbed by random noise. Second, the stochastic variational framework must be carefully designed in order to induce stochastic dynamics that possess Hamiltonian structures on Wasserstein manifold. Last by not the least, it is not clear at all that how to introduce the Christoffel symbol, a tool that plays the vital role in the typical kinetic dynamics, in the noise perturbed Wasserstein Hamiltonian flows on the density manifold.

To overcome the difficulties, we begin our study by investigating the classical Lagrangian functional perturbed by the Wong--Zakai approximation (see e.g. \cite{MR195142,MR0400425}) on phase space, and show that the critical point of the new Lagrangian functional is convergent to the known stochastic Hamiltonian flow driven by Wiener process. We further prove that the stochastic Wassertein Hamiltonian flow is the critical point of a stochastic variational principle (see e.g. \cite{MR2574753}). Meanwhile, the marco behaviors of this convergence indicates that the critical point of 
the marco Lagrangian functional corresponding to Wong--Zakai approximation is 
convergent to the stochastic Euler--Lagrange equation in density space.

Furthermore, a general variational principle is proposed to derive a large class of stochastic Hamiltonian equations on density manifold via Wong--Zakai approximation, such as stochastic nonlinear Schr\"odinger equation (see, e.g.,  \cite{PhysRevE.49.4627,PhysRevE.63.025601,MR1425880,UEDA1992166}), nonlinear Schr\"odinger equation with white noise dispersion (see, e.g., \cite{Agra01b,Agra01a}),  and the mean-field game system with common noise (see, e.g., \cite{MR3195844,MR3967062,MR3753660}). We would like to mention that although the Wong--Zakai approximation of stochastic differential equations has been studied for many years (see, e.g., \cite{MR195142,MR0400425,MR1313027,MR3712946}), few result is known for the convergence on the density manifold. In this work, we also provide some new convergence results of Wong--Zakai approximation for the continuity equation induced by stochastic Hamiltonian system and the stochastic Schr\"odinger equation on density space under suitable assumptions.

Another main message that we would like to convey in this paper is that the stochastic Hamiltonian flow on phase space, when viewed through the lens of conditional probability, induces the stochastic Wasserstein Hamiltonian flow on density manifold, and it is hard to observe those stochastic Hamiltonian structures
in the density manifold without the help of conditional probability.

The  organization of this article is as follows. In section 2, we review the formulation and derivation of Hamiltonian ODE, and use the Wong--Zakai approximation of the Lagrangian functional  to connect the classic and  stochastic variational principles on phase space. In section 3, we study the macro behaviors of  stochastic Hamiltonian ODE and its Wong--Zakai approximation, including the stochastic Euler--Lagrange equation on density space, Vlasov equation, as well as the generalized stochastic Wasserstein Hamiltonian flow. Several examples are demonstrated in section 4. Throughout this paper, we denote $C$ and $c$ as positive constants which may differ from line to line.

\section{Stochastic Hamiltonian ODEs}

In this section, we briefly review the classical and stochastic Hamiltonian flows on a finite dimensional Riemannian manifold.

The classical Hamiltonian flow on a smooth $d$-dimensional Riemannian manifold $(\mathcal M,g)$ with $g$ being the metric tensor of $\mathcal M,$ is derived by the variational problem 
$$I(x_0,x_T)=\inf_{(x(t))_{t\in [0,T]}}\{ \int_0^TL_0(x,\dot x)dt: x(0)=x_0,\; x(T)=x_T\}.$$
Here the Lagrangian $L_0$ is a functional (also called Lagrange action functional) defined on the tangent bundle of $\mathcal M$.
Its critical point induces the Euler-Lagrange equation
\begin{align*}
\frac {d}{dt} \frac {d}{d\dot x} L_0(x,\dot x)= \frac {d}{dx} L_0(x,\dot x).
\end{align*}
When $L_0(x,\dot x)=\frac 12 {\dot x}^{\top}g(x){\dot x}-f(x)$  with a smooth potential functional $f$ on $\mathcal M$, the Euler-Lagrange equation can be rewritten as a Hamiltonian system,
\begin{align*}
\dot x=g(x)^{-1}p,\;
\dot p=-\frac 12 p^{\top} d_x g^{-1}(x) p -d_x f(x)
\end{align*}
Here $\top$ denotes the transpose, 
$p=g(x)\dot x$ and the Hamiltonian is $$H_0(x,p)=\frac 12p^{\top}g^{-1}(x)p+f(x).$$
However, the Lagrange action functional $L_0(x,\dot x)$ may not be homogeneous or it can by impacted by random perturbations in some problems, which is the reason to introduce stochastic Hamiltonian flows. 

Let us start with the case that $L(x,\dot x)$ is composed by the deterministic Lagrange functional $L_0(x,\dot x)$ and a random perturbation $\eta \sigma (x)\dot \xi_{\delta}(t)$. Here $\xi_{\delta}$ can be chosen as a piecewise continuous differentiable function which obeys certain distribution law in a complete probability space $(\Omega,\mathbb F,\mathbb P)$ with a filtration $\{\mathbb F_t\}_{t\ge 0}$, 
$\sigma(x)$ is a potential function and $\eta\in\mathbb R$ characterizes the noise intensity.  In this paper, $\xi_{\delta}$ is taken as a Wong-Zakai approximation (see e.g. \cite{MR195142}) of the standard Brownian motion such that $\dot \xi_{\delta }$ is a real function. When $\delta \to 0,$ $\xi_{\delta}(t)$ is convergent to the Brownian motion $B(t)$ in pathwise sense or strong sense. For fixed $\omega\in \Omega, $ since $\xi_{\delta}(t)$ is a stochastic process on $(\Omega,\mathbb F,\mathbb P)$ with piecewise continuous trajectory, the value of the action functional $\int_0^TL_0(x,\dot x)-\eta\sigma (x)\dot \xi_{\delta}(t)dt$ is finite for any given $x(0)=x_0,x(T)=x_T.$

Newton's law can be used to derive the Euler--Lagrange equation or the Hamiltonian system in the stochastic case. In order to find out the critical point of $\int_0^TL_0(x,\dot x)-\eta\sigma (x)\dot \xi_{\delta}(t)dt,$ we calculate its  G\^{a}teaux derivative (see, e.g., \cite{MR0423094}). 
Set $x_{\epsilon}(t)=x(t)+\epsilon h(t),$ $h(0)=h(T)=0,$ the Newton's law indicates the critical point satisfies 
\begin{align*}
\frac {d}{dt} \frac {\partial}{\partial \dot x} L(x,\dot x)=\frac {\partial }{\partial x} L(x,\dot x)= \frac {\partial }{\partial x} L_0(x,\dot x)-\eta \frac {\partial }{\partial x}  \sigma(x) \dot  \xi_{\delta t},
\end{align*}
which is equivalent to the integral equation 
\begin{align*}
 \frac {\partial}{\partial \dot x} L(x(t),\dot x(t))-  \frac {\partial}{\partial \dot x} L(x(0),\dot x(0))=\int_0^t\frac {\partial }{\partial x} L_0(x,\dot x)ds-\eta \int_{0}^t \frac {\partial }{\partial x}  \sigma(x) d \xi_{\delta t}.
\end{align*}
One can also introduce the Legendre transformation $p=g(x)\dot x,$, and get
\begin{align}\label{inhs}
\dot x=g(x)^{-1}p,\;
\dot p=-\frac 12 p^{\top} d_x g^{-1}(x) p-d_x f(x)-\eta d_x \sigma(x) \dot \xi_{\delta }.
\end{align}
Since it can be rewritten as
\begin{align*}
\dot x= \frac {\partial }{\partial p}H_0(x,p)+\frac {\partial }{\partial p}H_1(x,p) \dot \xi_{\delta}, \; 
\dot p= -\frac {\partial }{\partial x}H_0(x,p)-\frac {\partial }{\partial x}H_1(x,p)\dot \xi_{\delta},
\end{align*}
where $H_1(x,p)=\sigma(x)$, the equations form a stochastic Hamiltonian system.

\begin{rk}
When $\dot \xi_{\delta }$ is a constant, the Hamilton's principle gives a Hamiltonian system with a homogenous perturbation. 
Otherwise, for a fixed $\omega,$ the Hamilton's principle leads to a Hamiltonian system with an inhomogenous perturbation. 
\end{rk}

\subsection{Wong--Zakai approximation in $\mathcal M=\mathbb R^d$}

In this part, we show that the limit of the Wong-Zakai approximation \eqref{inhs} is a stochastic Hamiltonain system.

\begin{lm}\label{bound-case}
Let $\mathcal M=\mathbb R^d$ and $T>0,$ $g$ be the identity matrix $\mathbb I_{d\times d}$.
Assume that $f,\sigma\in \mathcal C_b^2(\mathcal M),$ $\xi_{\delta}$ is the linear interpolation of $B(t)$ with width $\delta$ and that $x_0,p_0$ is $\mathbb F_0$-apdated. Then \eqref{inhs} on $[0,T]$ is convergent to 
\begin{align}\label{lim-sode0}
d x=p,\;
d p=-d_xf(x)-\eta d_x \sigma(x) \circ d B(t),\; \text{ a.s.},
\end{align}
where $\circ$ denotes the stochastic integral in the Stratorovich sense.
\end{lm}

\begin{proof}
The condition of $\sigma,f$ ensures the global existence of a unique strong solution for \eqref{inhs} and \eqref{lim-sode0} by using standard Picard iterations. Then one can follow the classical arguments (see e.g. \cite{MR0400425}) to show that the solution of \eqref{inhs} is convergent to that of \eqref{lim-sode0} and that the right hand side of \eqref{inhs} is convergent to that of \eqref{lim-sode0}. 
\end{proof}

The following lemma relaxes the classical conditions on the convergence of Wong--Zakai approximation whose proof is presented in Appendix. We call that
$g$ is equivalent to $\mathbb I_{d\times d}$ if $g\in \mathcal C_b^{\infty}(\mathbb R^d;\mathbb R^d)$ is symmetric satisfying  $\Lambda\mathbb{I}_{ d\times d} \succeq g(x)\succeq \lambda \mathbb{I}_{ d\times d}$ for some constant $0<\lambda\le \Lambda.$
In the following, we will use the standard notation for the matrix product, that is, $g(x)\cdot(y,z)=y^{\top} g(x) z$ and $g(x)\cdot y=g(x)y$.

\begin{lm}\label{rd-won}
Let $\mathcal M=\mathbb R^d$, $T>0$, $g$ be equivalent to $\mathbb I_{d\times d}$.
Assume that $f,\sigma\in \mathcal C^2_{p}(\mathcal M)$, $\xi_{\delta}$ is the linear interpolation of $B(t)$ with the width $\delta$, that $x_0,p_0$ are $\mathbb F_0$-apdated and possess any finite $q$-moment, $q\in \mathbb N^+$, and that 
\begin{align}\nonumber
& H_0(x,p)\ge c_0|p|+c_1 |x|, \; \text{for large enough} \; |x|,|p|\\\label{grow-con}
&\eta^2 |\nabla_{pp} H_0(x,p) \cdot (\nabla_x \sigma(x),\nabla_x \sigma(x))|
+ \eta |\nabla_{pp} H_0(x,p) \cdot (p,\nabla_x \sigma(x))|\\\nonumber
&+ \eta |\nabla_{pp} H_0(x,p) \cdot (\nabla_x \sigma(x), -\frac 12p^{\top }d_xg^{-1}(x)p-\nabla_x f(x))|+\eta|\nabla_{px} H_0(x,p)\cdot(\nabla_x \sigma,g^{-1}(x)p)| \\\nonumber
&\quad+ \eta |\nabla_{p} H_0(x,p) \cdot \nabla_{xx} \sigma(x)g^{-1}(x)p| \le C_1+c_1 H_0(x,p).
\end{align}
Then the solution of \eqref{inhs} on $[0,T]$ is convergent in probability to the solution of 
\begin{align}\label{lim-sode}
d x=g^{-1}(x)p,\;
d p=-\frac 12 p^{\top} d_x g^{-1}(x) p-d_x f(x)-\eta d_x \sigma(x) \circ d B(t).
\end{align}
\end{lm}

Denote the solution of \eqref{inhs} by $(x^{\delta}(\cdot,x_0,p_0),p^{\delta}(\cdot,x_0,p_0)).$
According to Lemma \ref{rd-won}, by studying the equation of $\frac {\partial}{\partial x_0} x^{\delta}(t,x_0,p_0)$ and $\frac {\partial}{\partial p_0} x^{\delta}(t,x_0,p_0)$, one could obtain the following convergence result.

\begin{cor}
Under the condition of Lemma \ref{rd-won}, let $f,\sigma\in \mathcal C_p^3 (\mathcal M).$ Then for any $\epsilon>0$, it holds that 
\begin{align*}
   &\lim_{\delta\to 0} \mathbb P\Big(\sup_{t\in [0,T]}|\frac {\partial}{\partial x_0} x^{\delta}(t,x_0,p_0)-\frac {\partial}{\partial x_0} x(t,x_0,p_0)|\\
   &+\sup_{t\in [0,T]}|\frac {\partial}{\partial p_0} x^{\delta}(t,x_0,p_0)-\frac {\partial}{\partial p_0} x(t,x_0,p_0)| \ge \epsilon \Big)=0. 
\end{align*}
\end{cor}

\begin{rk}
 One may impose more additional conditions on the coefficients $f,\sigma$ to obtain the strong convergence order $\frac 12$ of the Wong--Zakai approximation, that is,
 \begin{align*}
     &\E \Big [ \sup_{t\in[0,T]}|x^{\delta }(t)-x(t)|^p\Big ]
     +\E \Big [ \sup_{t\in[0,T]}|x^{\delta }(t)-x(t)|^p\Big ]
     \le C \delta^{\frac p2}.
 \end{align*}
 
 The convergence in probability yield that there exists a pathwise convergent subsequence. In this sense, the limit equation of \eqref{inhs} is \eqref{lim-sode} on $[0,T]$.
When the growth condition \eqref{grow-con} fails, one could also obtain the convergence in probability of $(x^{\delta},p^{\delta})$ before the stopping time $\tau_R\wedge \tau_{R_1}$ (see Appendix for the definition of $\tau_R$ and $\tau_{R_1}$).  One could also choose different type of Wong--Zakai approximation of the Wiener process and obtain similar results (see, e.g., \cite{MR195142}). 
\end{rk}

\subsection{Wong--Zakai approximation on a differential manifold $\mathcal M$}

Assume that $\mathcal M\subset \mathbb R^k$ is a   $d$-dimensional differential manifold of class $\mathcal C^{\alpha}, \alpha\in \mathbb N^+\cup {\infty}$ without boundary. Given a $\mathcal C^{\alpha}$-diffeomorphism $\phi: W \to V\subset \mathcal M$ from an open subset $W$ of $\mathbb R^d$  to an open set $V$ of $\mathcal M,$ the inverse  $\phi^{-1}: V\to W$ is called a chart or coordinate system on $\mathcal M.$  The coordinate components are denoted by $\Phi_1,\Phi_2,\cdots,\Phi_d$, $d\in \mathbb N^+$. 
The tangent bundle of $\mathcal M$ is denoted by $T\mathcal M:=\{(x,y)\in \mathbb R^k \times \mathbb R^k |x\in \mathcal M, y\in T_x(\mathcal M)\}.$ Moreover, $dim T_x(\mathcal M)=d.$
The canonical projection is denoted by $\pi: T \mathcal M \to \mathcal M.$


 In the following, we start from the deterministic Hamiltonian system
\begin{align*}
\dot x&=p,\;\\
\dot p&=-d_x f(x),
\end{align*}
where the vector field $(p,-d_x f(x))\in T_{(x,p)}T \mathcal M$ for all $(x,p)\in T \mathcal M.$
We show how the random force can be added to the system so that $(\dot x, \dot p)\in \mathbb R^{k}\times \mathbb R^{k}$ is still tangent to $T\mathcal M$ at $(x,p)$.
As a physical interpretation, this tangent condition
 represents the constrain of the motion equations and
  is to ensure that the physical motion is living in $T\mathcal M$ by the Kamke property of the maximal solutions (see e.g. \cite[Chapter 3]{MR1368671}).
Consider $\mathcal M$  which is regularly defined as the zero level set of a $\mathcal C^{\infty}$ map $F$ from $\mathbb R^k$ to $\mathbb R^{k-d}$. 
Then we have that  the tangent space to $\mathcal M$ at $x$ is $T_x\mathcal M:=\{p\in \mathbb R^k | F'(x)p=0 \}$, and $T\mathcal M=\{(x,p)\in \mathbb R^k \times \mathbb R^k | F(x)=0, F'(x)p=0\}.$ We can also obtain 
\begin{align*}
TT\mathcal M=\{(x,p,\dot x,\dot p)| F(x)=0, F'(x)p=0, F'(x)\dot x=0, F''(x)(\dot x,p)+F'(x)\dot p=0\}.    
\end{align*}
Therefore, if the added random force satisfies, 
\begin{align}\label{con-wel}
F'(x)\dot p =-F''(x)(\dot x,p)=\psi(x;p,\dot x), \; \dot x\in T_x(\mathcal M),
\end{align}
we have $(\dot x, \dot p)\in T_{(x,p)}(T\mathcal M)$.
Following \cite{MR1368671}, we denote a smooth mapping $\psi$ from the vector bundle $\{(x;u,v)\in \mathbb R^k\times(\mathbb R^k\times \mathbb R^k)| x\in \mathcal M, u,v\in T_x(\mathcal M)\}$  to $\mathbb R^{k-d}$.
Given any vector $z\in \mathbb R^{k-d}$, denote by $Az\in (Ker\; F'(x))^{\perp}=(T_x\mathcal M)^{\perp}$ the unique solution of $F'(x)\dot p=z$. Hence, the solution of \eqref{con-wel} satisfies
\begin{align*}
\dot p= \mu (x;p,\dot x)+w,
\end{align*}
where $\mu (x;p,\dot x)=A\psi(x;p,\dot x)\in (T_x(\mathcal M))^\perp$ and $w\in T_x(\mathcal M).$
We observe that to ensure $(\dot x,\dot p)\in T_{(x,p)}(T\mathcal M)$, it suffices to take $u,w\in T_x(\mathcal M)$ and define $(\dot x,\dot p)=(u,\mu (x;p,u)+w).$
In Eq. \eqref{inhs} with the driving noise being $-d_x\sigma(x)\dot \xi_{\delta},$ using the above condition, we can verify that it satisfies that $(\dot x,\dot p)\in T_{(x,p)}(T\mathcal M).$ Similarly, a second order differential equation with random force satisfies
\begin{align*}
\ddot x=\mu (x;\dot x,\dot x)+\mathcal R(t,x,\dot x),
\end{align*}
where  $\mathcal R_t: T\mathcal M \ni (x,\dot x) \mapsto \mathcal{R}(t,x,\dot x)\in  \mathbb R^k$ is a tangent vector field on $\mathcal M.$ A typical example is that {$\mathcal R=-\alpha \dot x+a(t,x)$ with the frictional force $-\alpha \dot x$}  and applied random force $a(t,x)=-d_x\sigma(x)\dot \xi_{\delta}(t).$ When $\mathcal R=0,$ the above equation is inertial and is so-called geodesic equation on $\mathcal M$, {which plays an important role in the optimal transport theory (see e.g. \cite{Vil09,MR2808856,CP17,CHLZ12})}.

\begin{lm}\label{mani-wong}
Suppose that $\mathcal M$ is a $d$-dimensional compact smooth differential manifold. Let $g=\mathbb I$, $f,\sigma$ be smooth functions on $\mathcal M$, $\xi_{\delta}$ be the linear interpolation of $B(t)$ with width $\delta$, and that $x_0,p_0$ are $\mathbb  F_0$-adapted and possess any finite $q$-moment, $q\in \mathbb N^+$. Then  $(x^{\delta},p^{\delta})$ is convergent in probability to the solution $(x,p)$ of \eqref{lim-sode}.  
\end{lm}

\begin{proof}
The existence and uniqueness of $(x,p)$ can be found in \cite{Hsu02}. 
The global existence of $(x^{\delta},p^{\delta})$ could be also obtained by the fact that $g=\mathbb I$, $f$ and $\sigma$ are globally Lipschitz
and that the growth condition \eqref{grow-con} holds.
We only need to show the convergence of $(x^{\delta},p^{\delta})$ in probability to $(x,p)$. 
Since $T\mathcal M$ is 2$d$-dimensional manifold which could be embedding to $\mathbb R^{2k},$ we can extend the vector field $V(x,p):=(p,-d_xf(x)-\eta d_x \sigma(x))$ to a vector field $\widetilde V(\cdot,\cdot)$ on $\mathbb R^{2k}$. And thus the equations of $(x,p)$ and $(x^{\delta},p^{\delta})$ can be viewed as the equations on $\mathbb R^{2k}$. 
The global existence of $(x,p)$ and $(x^{\delta},p^{\delta})$, together with Lemma \ref{rd-won}, yield the convergence in probability of $(x^{\delta},p^{\delta}).$
\end{proof}

\begin{rk}
The above result relies on the particular structure of $g=\mathbb I$ and the growth condition \eqref{grow-con}. If this condition \eqref{grow-con} fails, the explosion time $e(x^{\delta},p^{\delta})$ of $(x^{\delta},p^{\delta})$ may depend on $\delta$. And the convergence in probability may only hold before $e(x,p)\wedge \inf\limits_{\delta>0} e(x^{\delta},p^{\delta}).$ When applying different type of Wong--Zakai approximations, the different type of stochastic ODEs could be derived (see e.g. \cite{MR637061}).
\end{rk}

To end this section, we give a special example of stochastic Hamiltonian flows which concentrates on a submanifold with conserved quantities. 

\begin{ex}
Let $\mathcal M=\mathbb R^d,$ $g$ and $\widetilde g$ 
be metrics equivalent to $\mathbb I_{d\times d}.$ 
Define an action functional with random perturbation in dual coordinates, 
\begin{align*}
-\int_0^T (\<p,\dot x\>- H_0(x,p))dt+\int_0^TH_1(x,p)d\xi_{\delta}(t),
\end{align*}
where $H_0(x,p)=\frac 12 p^{\top}g^{-1}(x)p+f(x),$ $H_1(x,p)= \eta \frac 12 p^{\top}\widetilde g^{-1}(x)p+\eta \sigma(x)$ with smooth potentials $f$ and $\sigma$. 
Then the critical points under the constrain $x(0)=x_0,x(T)=x_T$ satisfies the stochastic Hamiltonian flows
\begin{align*}
\dot x&=\frac {\partial H_0}{\partial p}(x,p)+ \frac {\partial H_1}{\partial p}(x,p) \dot \xi_{\delta },\;\\
\dot p&=-\frac {\partial H_0}{\partial p}(x,p)- \frac {\partial H_1}{\partial p}(x,p)\dot \xi_{\delta }.
\end{align*}
Its limit  $(x,p)$ lie on the manifold $\{H_0(x,p)=H_0(x_0,p_0), H_1(x,p)= H_1(x_0,p_0)\}$ when the Hamiltonians satisfies that $\{H_0,H_1\}=0$ with $\{\cdot,\cdot\}$ being the Possion bracket. 
 Similar to Lemma \ref{rd-won}, it can be shown that $(x^{\delta},p^{\delta})$ converges globally to $(x,p)$ in probability if $H_0$ or $H_1$ satisfies the growth condition \eqref{grow-con}.

\end{ex}


\section{Stochastic Wasserstein Hamiltonian flow}
\label{swhf}
In this section, we study the behaviors of the inhomogenous Hamiltonian system \eqref{inhs} and stochastic Hamiltonian system \eqref{lim-sode} on the density manifold. 
To illustrate the strategy, let us focus on the case that $(\mathcal M,g)$ equals $(\mathbb T^d,\mathbb I)$ or $(\mathbb R^d,\mathbb I).$
Given the filtered complete probability space $(\Omega,\mathbb F, (\mathbb F_t)_{t\ge 0},\mathbb P),$ we assume that $\xi_{\delta}(t)$ is the piecewisely linear Wong--Zakai approximation of a standard Brownian motion, and that $x_0$ is a random variable with the density $\rho_0$ on another complete probability space $(\widetilde \Omega,\mathcal {\widetilde  F},\widetilde {\mathbb P})$. 
For a fixed $\widetilde \omega\in \widetilde \Omega$, we denote $\tau^{\delta}:=\inf\{t\in (0,T] | x_t^{\delta} \; \text{is not a smooth diffeomorphism on} \; \mathcal M\}$, $p_t^{\delta }=v(t,x_t^{\delta})$ is the vector field depending on the position and time. 
{Here we view the momentum $p$ as the function $v$ depending on both time and space.}
Eq. \eqref{inhs} becomes
\begin{align*}
&\frac {d}{dt} x_t^{\delta}=v(t, x_t^{\delta}),\\
&\frac {d}{dt} v(t,x_t^{\delta})=-\nabla f(x_t^{\delta})-\eta \nabla \sigma(x_t^{\delta}) \dot \xi_{\delta}(t).
\end{align*}
Differentiating $v(t,x_t^{\delta}(x_0))$ before $\tau^{\delta}$ leads to
\begin{align*}
\partial_t v(t,x_t^{\delta}(x_0))+\nabla v(t,x_t^{\delta}(x_0))\cdot \frac {d}{dt} x_t^{\delta}
&=\partial_t v(t,x_t^{\delta}(x_0))+\nabla v(t,x_t^{\delta}(x_0))\cdot v(t,x_t^{\delta}(x_0))\\
&=-\nabla f(x_t^{\delta}(x_0))-\eta \nabla \sigma(x_t^{\delta}(x_0)) \dot \xi_{\delta}(t).
\end{align*}
Taking $x_0=(x_t^{\delta})^{-1}(x)$, we obtain the following conservation law with random perturbation, 
\begin{align}\label{con-law}
\partial_t v(t,x)+\nabla v(t,x)\cdot v(t,x)=-\nabla f(x)-\eta \nabla \sigma(x) \dot \xi_{\delta}(t).
\end{align}
Taking any test function $\psi$ in $C^{\infty}(\mathcal M),$ it holds that 
\begin{align*}
\frac {d}{dt}\E_{\widetilde \Omega}[\psi(x_t^{\delta}(x_0))]
&=\frac {d}{dt} \int_{\mathcal M}\psi(x)\rho(t,x)dx
=\int_{\mathcal M}\nabla \psi(x_t^{\delta}(x)) \cdot v(t,x_t^{\delta}(x))\rho_0(x)dx\\
&=\int_{\mathcal M}\nabla \psi(x) \cdot v(t,x)\rho_t(x)dx,
\end{align*}
which implies that for  $\omega\in \Omega,$ $\rho_t=x_t^{\delta}{\#}\rho_0$, i.e., $\rho_t$ equals  the distribution generated by the map $x_t(\cdot)$ push-forward $\rho_0$,  satisfies the continuity equation,
\begin{align}\label{cont}
\partial_t \rho(t,x)+\nabla\cdot (\rho(t,x)v(t,x))=0.
\end{align}

Introducing the pseudo inverse $(-\Delta_{\rho})^\dagger$ (see e.g. \cite{CLZ20}) of $-\Delta_{\rho}=-\nabla \cdot (\rho \nabla)$ for a positive density $\rho$, we denote $S_t=(-\Delta_{\rho_t})^\dagger \partial_t \rho_t.$ When there exists a potential $S$ such that $v=\nabla S$, the conservation law with random influence \eqref{con-law} and the continuity equation \eqref{cont} induce a Hamiltonian system in density manifold before $\tau^{\delta}$,
\begin{align}\label{whs}
\partial_t \rho_t&=\frac {\delta }{\delta S_t} \mathcal H_0(\rho_t,S_t)=-\nabla\cdot (\rho_t\nabla S_t),\\\nonumber
\partial_t S_t&=-\frac {\delta }{\delta \rho_t} \mathcal H_0(\rho_t,S_t)-\frac {\delta}{\delta \rho_t} \mathcal H_1(\rho_t,S_t)\dot \xi_{\delta}(t)+C(t)\\\nonumber
&=-\frac 12 |\nabla S_t|^2-\frac {\delta }{\delta \rho_t}\mathcal F(\rho_t)- \frac {\delta }{\delta \rho_t}\eta \Sigma (\rho_t) \dot \xi_{\delta}(t) +C(t),
\end{align}
where $C(t)$ is  an arbitrary stochastic process on $(\Omega,\mathbb F,\mathbb P)$ independent of the spatial position $x$ and  $v(0,\cdot) = \nabla S(0,\cdot)$.
Here the dominated average energy is $$\mathcal H_0(\rho,S):=K(\rho,S)+\mathcal  F(\rho)=\int_{\mathcal M}\frac 12 |\nabla S(x)|^2\rho(x) dx+\int_{\mathcal M} f(x)\rho(x)dx,$$ and  the 
perturbed average energy is $$\mathcal H_1(\rho,S,t)= \eta \Sigma (\rho_t)=\eta\int_{\mathcal M}  \sigma(x)\rho(x)dx.$$

Taking $\delta \to 0,$ the limit system becomes a stochastic Hamiltonian system,
\begin{align}\label{lim-shs}
d \rho_t&=\frac {\delta }{\delta S_t} \mathcal H_0(\rho_t,S_t)dt,\\\nonumber
d S_t&=-\frac {\delta }{\delta \rho_t} \mathcal H_0(\rho_t,S_t)-\frac {\delta}{\delta \rho_t} \mathcal H_1(\rho_t,S_t)\star d\xi+C(t)dt,
\end{align}
where $\xi$ is the limit process of $\xi_{\delta}$ in path-wise sense. We would like to remark that the solution of \eqref{lim-shs} is not predictable in general.  
In our particular case, since $\xi_{\delta}(t)$ is a piecewisely linear Wong-Zakai approximation of $B(t)$, the limit of  \eqref{con-law}, \eqref{cont} is the following system in Stratonovich sense,
\begin{align}\label{lim-whs}
&d\rho_t=-\nabla\cdot (\rho(t,x)v(t,x))dt,\\\nonumber
&dv(t,x)+\nabla v(t,x)\cdot v(t,x)dt=-\nabla f(x)dt-\eta \nabla \sigma(x) \circ dB_t.
\end{align}
We would like to emphasize that the above analysis indicates a principle for deriving the stochastic Hamiltonian system on Wasserstein manifold: {\it The conditional probability density of stochastic Hamiltonian flow in phase space is a stochastic Hamiltonian flow in density manifold almost surely.} 
In the following we always assume that the initial distribution $\rho(0,\cdot)$ of $x_0$ and the initial velocity $v(0,\cdot)$ are smooth and bounded. 

\begin{prop}\label{path-con-rhos}
Suppose that $\mathcal M$ is a d-dimensional compact smooth differential submanifold and $T>0$. Let $g=\mathbb I,$ $v(0,\cdot)$ be a smooth vector field, $f,\sigma$ be smooth function on $\mathcal M$, $\xi_{\delta}$ be the linear interpolation of $B(t)$ with width $\delta$, and that $x_0,p_0$ are $\mathbb F_0$-adapted and possess any finite $q$-moment, $q\in \mathbb N^+$. Then there exists a stopping time  $\tau$ such that there exists a subsequence of   $(\rho^{\delta},v^{\delta})$ which converges in probability to the solution $(\rho,v)$ of \eqref{lim-whs}   before  $\tau.$
\end{prop}

\begin{proof}
Applying Lemma \ref{mani-wong}, we have that $(x^{\delta}_t,v(t,x_t^{\delta}))$ is convergent to $(x_t,v(t, $ $ x_t))$ in $[0, T]$, a.s., up to a subsequence. 
Define the stopping time $\tau=\inf\{t\in (0,T] |\, x_t \,$ is not smooth diffeomorphism on $ \mathcal M\}.$
For convenience, let us take a subsequence such that for almost $\omega \in \Omega,$ $(x^{\delta}_t,v(t,x_t^{\delta}))$ converges to $(x_t,v(t,x_t))$ and $\frac {\partial}{\partial x_0}x^{\delta}_t(x_0)$ convergences to $\frac {\partial}{\partial x_0}x_t(x_0).$ Before $\tau(\omega),$ there exists $\alpha >0$ such that $\det (\frac {\partial}{\partial x_0}x_t^{-1}(x_0))>\alpha.$ The pathwise convergence of $x^{\delta}$ implies that for any $\epsilon>0$ there exists $\delta_0=\delta(\epsilon,\omega)>0$ such that when $\delta \le \delta_0$, 
$\det (\frac {\partial}{\partial x_0}(x_t^{\delta})^{-1}(x_0))>\alpha-\epsilon>0.$
Notice that the density function $\rho^\delta(t,y)$ of $x^{\delta}_t$ satisfies $\rho^{\delta}(t,y)=|\det(\nabla x^{\delta}_t(y))|\rho(0,x^{\delta}_t(y)).$ Since $\rho(0,\cdot)$ is smooth for any fixed $\omega$ and the pathwise convergence of $x^{\delta}$ holds, it follows that $\rho^{\delta}(t,y)$ converges to 
the density function of $x_t$, which is $\rho(t,y)=|\det( \nabla x_t(y))|\rho(0,x_t(y)).$
Similarly, the pathwise convergence of $v^{\delta}(t,x^{\delta}_t(y))$ to $v(t,x_t(y)),$
together with invertibility of $x^{\delta}_t$ and $x_t$, implies the convergence of $v^{\delta}(t,x)$ to $v(t,x).$ 
Consequently, the solution of $(\rho^{\delta},v^{\delta})$ is convergent to $(\rho, v)$ in pathwise sense up to a subsequence.
\end{proof}

\subsection{Vlasov equation}
We would like to present the connections and differences between the classic Vlasov equation and stochastic Wasserstein Hamiltonian flow in this part.
For simplicity, let us consider the case that $\mathcal M=\mathbb R^d.$
We fix $\widetilde \omega\in \widetilde \Omega$, and consider \eqref{inhs}.
Taking differential on $\E_{\Omega}[\phi(x_t^{\delta},p_t^{\delta})]$ where $\phi$ is a sufficient smooth test function,  we get 
\begin{align*}
\frac{d}{dt} \E_{\Omega}[\phi(x_t^{\delta},p_t^{\delta})]
&=\E_{\Omega}[\nabla_x \phi(x_t^{\delta},p_t^{\delta})\frac{d}{dt}  x^{\delta}_t+\nabla_p \phi (x_t^{\delta},p_t^{\delta})\frac{d}{dt}  p^{\delta}_t]\\
&=\E_{\Omega}[\nabla_x \phi(x_t^{\delta},p_t^{\delta})p_t+\nabla_p \phi (x_t,p_t)(-\nabla_x f(x_t^{\delta})-\eta \nabla_x \sigma(x_t^{\delta}) \dot \xi_{\delta })].
\end{align*}
Denoting the initial joint probability density function by $F_0(x,p)$, it
 holds that
\begin{align*}
&\frac{d}{dt}  \int_{\mathbb R^d \times \mathbb R^d} \phi(x_t^{\delta},p_t^{\delta})F_0(x,p) dxdp\\
&= \int_{\mathbb R^d \times \mathbb R^d} \Big(\nabla_x \phi(x_t^{\delta },p_t^{\delta}) p^{\delta}_t +\nabla_p \phi (x_t^{\delta},p_t^{\delta}) (-\nabla_x f(x_t^{\delta})-\eta \nabla_x \sigma(x_t^{\delta}) \dot \xi_{\delta})\Big)F_0(x,p) dxdp
\end{align*}
Thus the joint distribution on $\Omega$,  $F_t^{\delta}=(x_t^{\delta},p_t^{\delta})^{\#}F_0$, satisfies 
\begin{align*}
&\int_{\mathbb R^d \times \mathbb R^d} \phi(x,p) \frac{d}{dt}  F_t^{\delta}(x,p) dxdp\\
&=\int_{\mathbb R^d \times \mathbb R^d} \Big(\nabla_x \phi(x,p) p+\nabla_p \phi (x,p) (-\nabla_x f(x)\Big)F_t(x,p) dxdp\\
&+ \E_{\Omega}[\nabla_p \phi (x_t^{\delta},p_t^{\delta}) (-\eta \nabla_x \sigma(x_t^{\delta}))\dot \xi_{\delta}(t)\Big]. 
\end{align*} 


Notice that the solution process $x^{\delta}_{t}$ is $\mathbb F_{t_{k+1}}$-measurable when $t\in (t_k,t_{k+1}], t_k=k\delta t$ and $\mathbb F_{t_k}$-measurable when $t=t_k$, and  $x_{t}$ is $\mathbb F_{t}$-measurable. 
By applying the chain rule, we have that for $t\in (t_k,t_{k+1}],$
\begin{align*}
&\int_0^t \E_{\Omega}[\nabla_p \phi (x_s^{\delta},p_s^{\delta}) (-\eta \nabla_x \sigma(x_s^{\delta}))\dot \xi_{\delta}(s)\Big]ds\\
&=\sum_{j=0}^{k-1} \int_{t_j}^{t_{j+1}} \E_{\Omega}[\nabla_p \phi (x_s^{\delta},p_s^{\delta}) (-\eta \nabla_x \sigma(x_s^{\delta}))\dot \xi_{\delta}(s)\Big]ds\\
&+\int_{t_k}^t\E_{\Omega}[\nabla_p \phi (x_s^{\delta},p_s^{\delta}) (-\eta \nabla_x \sigma(x_s^{\delta}))\dot \xi_{\delta}(s)\Big]ds\\
&=\sum_{j=0}^{k-1} \int_{t_j}^{t_{j+1}} \E_{\Omega}[\nabla_p \phi (x^{\delta}_{t_j},p^{\delta}_{t_j}) (-\eta \nabla_x \sigma(x^{\delta}_{t_j}))\frac {B_{t_{j+1}}-B_{t_j}}{\delta}\Big]ds\\
&+\sum_{j=0}^{k-1} \int_{t_j}^{t_{j+1}} \E_{\Omega}\Big[ \Big(\nabla_p \phi (x_s^{\delta},p_s^{\delta}) (-\eta \nabla_x \sigma(x_s))-\nabla_p \phi (x_{t_j}^{\delta},p_{t_j}^{\delta}) (-\eta \nabla_x \sigma(x_{t_j}^{\delta}))\Big)\frac {B_{t_{j+1}}-B_{t_j}}{\delta}\Big]ds \\
&+\int_{t_k}^t\E_{\Omega}[\nabla_p \phi (x_{t_k}^{\delta},p_{t_k}^{\delta}) (-\eta \nabla_x \sigma(x^{\delta}_{t_k}))\frac {B_{t_{k+1}}-B_{t_k}}{\delta} \Big]ds\\
&+\int_{t_k}^t\E_{\Omega}[\Big(\nabla_p \phi (x^{\delta}_{s},p^{\delta}_{s}) (-\eta \nabla_x \sigma(x_s^{\delta})) -\nabla_p \phi (x^{\delta}_{t_k},p^{\delta}_{t_k}) (-\eta \nabla_x \sigma(X^{\delta}_s))\Big)\frac {B_{t_{k+1}}-B_{t_k}}{\delta} \Big]ds
\end{align*}

Then repeating similar arguments in the proof of Lemma \ref{rd-won},  we have that 
\begin{align*}
&\int_0^t \E_{\Omega}[\nabla_p \phi (x_t^{\delta},p_t^{\delta}) (-\eta \nabla_x \sigma(X_t^{\delta}))\dot \xi_{\delta}(t)\Big]ds\\
&=  \int_0^t \E_{\Omega}[\nabla_p \phi (x^{\delta}_{[t]_{\delta} \delta },p^{\delta}_{[t]_{\delta} \delta}) (-\eta \nabla_x \sigma(x^{\delta}_{[t]_{\delta} \delta}))\dot \xi_{\delta}(t)\Big]ds\\
&+\int_0^t \frac 12\E_{\Omega}[(\Delta_{pp} \phi (X^{\delta}_{[t]_{\delta} \delta },p_{[t]_{\delta} \delta})(-\eta \nabla_x \sigma(x^{\delta}_{[t]_{\delta} \delta}) )(-\eta \nabla_x \sigma(x^{\delta}_{[t]_{\delta} \delta})) (\dot \xi_{\delta}(t))^2\Big] ds\\
&+ o(\delta^{\beta}),
\end{align*}
where $\beta \in (0,\frac 12)$. 
Taking $\delta \to 0$ yield that the second order Vlasov equation 
\begin{align*}
\partial_t F(t,x,p)&=-\nabla_x\cdot (F(t,x,p)\frac {\partial H_0}{\partial p})+\nabla_p\cdot (F(t,x,p)\frac {\partial H_0}{\partial x})\\
&+\frac 12 \Delta_{pp}F(t,x,p)\cdot (\frac {\partial H_1}{\partial x},\frac {\partial H_1}{\partial x}).
\end{align*} 
This implies that when we consider the joint distribution on $\Omega$, the density function satisfies the second order Vlasov equation. However, when we consider the conditional probability on $\widetilde \Omega$ instead of $ \Omega,$ the conditional joint probability of Wong--Zakai approximation satisfies the following first order Vlasov equation, 
{
\begin{align*}
\partial_t F^{\delta}(t,x,p)&=-\nabla_x\cdot (F^{\delta}(t,x,p)\frac {\partial H_0}{\partial p})+\nabla_p\cdot (F^{\delta}(t,x,p)\frac {\partial H_0}{\partial x})\\
&+\nabla_p\cdot  (F^{\delta}(t,x,p)\frac {\partial H_1}{\partial x})\dot \xi_{\delta}.
\end{align*}
}
Its limit equation becomes 
\begin{align*}
d F(t,x,p)&=-\nabla_x\cdot (F(t,x,p)\frac {\partial H_0}{\partial p})dt+\nabla_p\cdot (F(t,x,p)\frac {\partial H_0}{\partial x})dt\\
&+\nabla_p\cdot  (F(t,x,p)\frac {\partial H_1}{\partial x})\circ dB_t.
\end{align*}

\subsection{Stochastic Euler-Lagrange equation in density space} \label{subsec:SEL}

In this section, we consider the Wasserstein Hamiltonian flow with random perturbation, i.e., the second order stochastic Euler-Lagrange equation from the Lagrange functional on density manifold.  The density space $\mathcal P(\mathcal M)$ is defined by 
\begin{align*}
    \mathcal P(\mathcal M)=\{\rho d vol_{\mathcal M}| \rho \in \mathcal C^{\infty}(\mathcal M), \rho\ge 0,\int_{\mathcal M}\rho dvol_{\mathcal M}=1\}.
\end{align*}
Its interior of $\mathcal P(\mathcal M)$ is denoted by $\mathcal P_{o}(\mathcal M).$ The tangent space at $\rho\in \mathcal P_{o}(\mathcal M)$ is defined by 
\begin{align*}
T_{\rho} \mathcal P_{o}(\mathcal M)=\{\kappa \in \mathcal C^{\infty}(\mathcal M)|\int_{\mathcal M} \kappa dvol_{\mathcal M}=0 \}.
\end{align*}
Define the quotient space of smooth functions $\mathscr F(\mathcal M)/\mathbb R=\{[\Phi]| \Phi\in \mathcal C^{\infty}(\mathcal M)\}$,
where $[\Phi]=\{\Phi+c|c\in \mathbb R\}.$ Then one could identify the element in  $\mathscr F(\mathcal M)/\mathbb R$ as the tangent vector in $T_{\rho}\mathcal P_{o}(\mathcal M)$ by using the map $\Theta: \mathscr F(\mathcal M)/\mathbb R \to T_{\rho}\mathcal P_{o}(\mathcal M), \; \Theta_{\Phi}=-\nabla \cdot (\rho \nabla \Phi).$ The boundaryless condition of $\mathcal M$ and the property of elliptical operator $\Delta_{\rho}(\cdot)=-\nabla \cdot (\rho \nabla (\cdot))$ ensures that $\Psi$ is one to one and linear. This implies that $\mathscr F(\mathcal M)/{\mathbb R}\cong T_{\rho}^* \mathcal P_{o}(\mathcal M)$, where $T_{\rho}^* \mathcal P_{o}(\mathcal M)$ is the cotangent space of $\mathcal P_{o}(\mathcal M)$.
$L^2$-Wasserstein metric on density manifold 
$g_W: T_{\rho}\mathcal P(\mathcal M)\times T_{\rho}\mathcal P(\mathcal M) \to \mathbb R$ is define by 
\begin{align*}
g_W(\kappa_1,\kappa_2)=\int_{\mathcal M} \<\nabla \Phi_1,\nabla \Phi_2\>\rho dvol_{\mathcal M}=\int_{\mathcal M} \kappa_1 (-\Delta_{\rho})^{\dagger} \kappa_2 dvol_{\mathcal M},
\end{align*}
where $\kappa_1=\Theta_{\Phi_1}$, $\kappa_2=\Theta_{\Phi_2},$ and $(-\Delta_{\rho})^{\dagger}$ is the pseudo inverse operator of $-\Delta_{\rho}.$ 
In deterministic case, it is known that 
the critical point of the Wasserstein metric $$\frac 12 W^2(\rho_0,\rho^1):=\inf_{\rho_t \in \mathcal P_o(\mathcal M)}\Big\{\int_0^1\int_{\mathcal M} \frac 12 g_W(\partial_t\rho_t, \partial \rho_t)dvol_{\mathcal M}dt \Big\}$$  satisfies the geodesic equation in cotangent bundle on density manifold (see e.g. \cite{CLZ20a}), that is,
 \begin{align*}
 &\partial_t\rho_t = -\nabla \cdot(\rho_t \nabla\Phi_t ),\;\\
 &\partial_t \Phi_t = -\frac 12 |\nabla \Phi_t|^2+C_t,
 \end{align*}
 where $\Phi_t=(-\Delta_{\rho_t})^{\dagger}\partial_t \rho_t,$ $C_t$ is independent of $x\in \mathcal M $. 
 The above geodesic equation in primal coordinates is the Euler--Lagrange equation,  
 \begin{align*}
\partial_t \frac {\delta}{\delta\partial_t \rho_t}\mathcal L(\rho_t,\partial_t \rho_t)=\frac {\delta}{\delta \rho_t} \mathcal L(\rho_t,\partial_t \rho_t)+C(t),
 \end{align*}
 where $\mathcal L(\rho_t,\partial_t\rho_t)=\frac 12g_W(\partial_t \rho_t,\partial_t \rho_t).$
 
Next, we consider the Lagrangian in density manifold with random perturbation, 
 \begin{align*}
   \mathcal L(\rho_t,\partial_t \rho_t)=\frac 12g_W(\partial_t \rho_t,\partial_t \rho_t)-\mathcal F(\rho_t)- \Sigma(\rho_t)\dot \xi_{\delta}(t), 
 \end{align*}
 and its variational problem $I_{\delta}(\rho^0,\rho^T)=\inf\limits_{\rho_t}\{\int_0^T\mathcal L(\rho_t,\partial_t \rho_t)dt| \rho_0=\rho^0,\rho_T=\rho^T\}.$
 
 \begin{tm}\label{tm-wong}
 The Euler Lagrangian equation of the variational problem $I_{\delta}(\rho^0,\rho^T)$ satisfies 
 \begin{align}\label{eul-lag}
 \partial_{tt}\rho_t+\Gamma_W(\partial_t\rho_t, \partial_t\rho_t)=-grad_W \mathcal F(\rho_t)-grad_W \Sigma (\rho_t) \dot \xi_{\delta},
 \end{align}
 where $grad_W \mathcal F(\rho_t)=-\nabla\cdot(\rho_t\nabla \frac {\delta}{\delta \rho_t}\mathcal F(\rho_t)),$ 
 $\Gamma_W(\partial_t\rho_t, \partial_t\rho_t)=\Delta_{\partial_t \rho_t} (-\Delta_{\rho_t})^\dagger\partial_t \rho_t 
 +\frac 12 \Delta_{\rho_t}|\nabla (-\Delta_{\rho_t})^\dagger\rho_t|^2.$
 Furthermore, Eq. \eqref{eul-lag} can be formulated as the following Hamiltonian system
 \begin{align}\label{hs-rhos}
& \partial_t \rho_t+\nabla \cdot (\rho_t\nabla \Phi_t)=0,\\\nonumber
 &\partial_t \Phi_t+\frac 12|\nabla \Phi_t|^2=-\frac {\delta }{\delta \rho_t} \mathcal F(\rho_t)-\frac {\delta }{\delta \rho_t} \Sigma(\rho_t)\dot \xi_{\delta},
 \end{align}
 where $\Phi_t=(-\Delta_{\rho_t})^{\dagger}\partial_t \rho_t$ up to a spatially constant stochastic process shift. 
 \end{tm}
 
 \begin{proof}
Consider a smooth perturbation $\epsilon h_t$ satisfying  $\int_{\mathcal M}h_tdvol_{\mathcal M}=0,$ $t\in [0,T]$ and $h_0=h_T=0.$ 
Applying Taylor expansion with respect $\epsilon$ and integration by parts, using $h_0=h_T=0$ and the fact that $\mathcal M$ is compact, we get 
\begin{align*}
&\int_{0}^T\mathcal L(\rho_t+\epsilon h_t, \partial_t\rho_t+\epsilon \partial_t h_t)dt  \\
&= \int_{0}^T\mathcal L(\rho_t, \partial_t\rho_t)dt
+\epsilon \int_0^T\int_{\mathcal M}(\frac {\delta}{\delta \rho_t} \mathcal L(\rho_t,\partial_t\rho_t)-\partial_t\frac {\delta}{\delta \partial_t \rho_t} \mathcal L(\rho_t,\partial_t\rho_t))\cdot h_t dvol_{\mathcal M}dt+o(\epsilon).
\end{align*}
Direct calculations lead to 
\begin{align*}
&\partial_t \frac {\delta }{\delta \partial_t \rho_t}\mathcal L(\rho_t,\dot \rho_t)=\partial_t((-\Delta_{\rho_t})^{\dagger}\partial_t \rho_t)\\
&\qquad \qquad\qquad \quad = (-\Delta_{\rho_t})^{\dagger}\partial_{tt}\rho_t-(-\Delta_{\rho_t})^{\dagger}(-\Delta_{\partial_t \rho_t} )^{{\dagger}}\partial_t \rho_t,\\
&\frac{\delta}{\delta \rho_t}\mathcal L(\rho_t,\dot \rho_t)= -\frac 12\nabla |(-\Delta_{\rho_t}^{\dagger})\partial_t \rho_t|^2-\frac {\delta }{\delta \rho_t} \mathcal F(\rho_t)-\frac {\delta }{\delta \rho_t} \Sigma(\rho_t)\dot \xi_{\delta}(t),
\end{align*}
which, together with the property $\int_{\mathcal M}h_tdvol_{\mathcal M}=0,$ yields  \eqref{eul-lag} up to a spatially-constant stochastic process shift by multiplying $\Delta_{\rho_t}$ on both sides. By introducing the Legendre transformation $\Phi_t=(-\Delta _{\rho_t})^{\dagger}\partial \rho_t,$ we obtain Eq. \eqref{hs-rhos} from Eq. \eqref{eul-lag}. 

\end{proof}

\begin{prop}\label{prop-sto}
The Euler--Lagrange equation of the variational problem $I(\rho^0,\rho^T)$,
\begin{equation*}
  I(\rho_0, \rho_T) = \int_0^T (\frac 12g_W(\partial_t \rho_t,\partial_t \rho_t)-\mathcal F(\rho_t))dt- \int_0^T \Sigma(\rho_t)\circ dB(t) 
\end{equation*}
satisfies 
 \begin{align}\label{eul-lag-sto}
 \partial_{tt}\rho_t+\Gamma_W(\partial_t\rho_t, \partial_t\rho_t)=-grad_W \mathcal F(\rho_t)-grad_W \Sigma (\rho_t) \circ dB_t,
 \end{align}
 where  $\rho_t$ is $\mathbb F_t$-measurable.
 Furthermore, Eq. \eqref{eul-lag-sto} can be formulated as the following  Hamiltonian system
 \begin{align}\label{hs-rhos-sto0}
& \partial_t \rho_t+\nabla \cdot (\rho_t\nabla \Phi_t)=0,\\\nonumber
 &\partial_t \Phi_t+\frac 12|\nabla \Phi_t|^2=-\frac {\delta }{\delta \rho_t} \mathcal F(\rho_t)-\frac {\delta }{\delta \rho_t} \Sigma(\rho_t) \circ dB_t,
 \end{align}
 where $\Phi_t=(-\Delta_{\rho_t})^{\dagger}\partial_t \rho_t$ up to a spatially constant stochastic process shift. 
\end{prop}

\begin{proof}
 Consider a smooth perturbation $\epsilon h_t$ satisfying  $\int_{\mathcal M}h_tdvol_{\mathcal M}=0,$ $t\in [0,T]$ and $h_0=h_T=0.$ 
 Notice that there exists $\Phi_t=(-\Delta_{\rho})^{\dagger}\partial_t \rho_t.$
 Using the equivalence of stochastic integral between It\^o sense and Stratonovich sense {(see e.g. \cite{MR1214374})}, we have that 
 \begin{align*}
 &\int_{0}^T\frac 12g_W(\partial_t\rho_t+\epsilon h_t,\partial_t\rho_t+\epsilon h_t)-\mathcal F(\rho_t+\epsilon h_t)dt-\int_0^T \Sigma(\rho_t+\epsilon h_t)dB_t\\
 &=\int_0^T \mathcal L_0(\rho_t,\partial_t \rho_t)dt + \int_0^T\Sigma(\rho_t) dB_t \\ 
 &+\epsilon \int_0^T\int_{\mathcal M}(\frac {\delta}{\delta \rho_t} \mathcal L_0(\rho_t,\partial_t\rho_t)-\partial_t\frac {\delta}{\delta \partial_t \rho_t} \mathcal L_0(\rho_t,\partial_t\rho_t))\cdot h_t dvol_{\mathcal M}dt\\
 &+\epsilon \int_0^T \int_{\mathcal M} \frac {\delta}{\delta \rho_t} \Sigma(\rho_t) \cdot h_t dvol_{\mathcal M} dB_t+o(\epsilon).
 \end{align*}
Similar to the proof of Theorem \ref{tm-wong}, we obtain \eqref{eul-lag-sto}
and its equivalent Hamiltonian system \eqref{hs-rhos-sto0}.
\end{proof}



\subsection{Generalized stochastic Wassersetin--Hamiltonian flow}

 In the last section, we show that the density of a Hamiltonian ODE with random perturbation satisfies the stochastic Wasserstein Hamiltonian flow. In this section, We derived the stochastic Wasserstein Hamiltonian flow via the random perturbation in the dual coordinates in density space. 
It provides a more general framework that can derive a large class of stochastic Wasserstein Hamiltonian flows which can not be obtained from the classic dynamics with perturations.

We introduce the following variational problem
 \begin{align}\label{gen-var-pri}
      I_{\delta}(\rho^0,\rho^T)=\inf\{\mathcal S(\rho_t,\Phi_t)| (-\Delta_{\rho_t})^{\dagger}\Phi_t \in \mathcal T_{\rho_t} \mathcal P_{o}(\mathcal M),\rho(0)=\rho^0, \rho(T)=\rho^T\}
 \end{align}
whose action functional is given by the dual coordinates, 
\begin{align*}
\mathcal S(\rho_t,\Phi_t)&=-\int_0^T \< \Phi(t),\partial_t \rho_t\>-\mathcal H_0(\rho_t, \Phi_t) dt +\int_0^T \mathcal H_1(\rho_t,\Phi_t)d\xi_{\delta}(t).
\end{align*}
Here $\mathcal H_0(\rho_t, \Phi_t)=\int_{\mathcal M}\frac 12 |\nabla \Phi_t|^2\rho_t d vol_{\mathcal M}+\mathcal F(\rho_t)$, $\mathcal H_1(\rho_t,\Phi_t)=\eta \int_{\mathcal M}\frac 12 |\nabla \Phi_t|^2\rho_t d vol_{\mathcal M}+\eta \Sigma (\rho_t)$, $\mathcal F$ and $\Sigma$ are smooth potential functions.

\begin{tm}\label{tm-sto-mix}
 The critical point of the variational problem $I_{\delta}(\rho^0,\rho^T)$ satisfies 
 the following  Hamiltonian system
 \begin{align}\label{hs-rhos-mix}
& \partial_t \rho_t+\nabla \cdot (\rho_t\nabla \Phi_t)+\eta \nabla \cdot (\rho_t\nabla \Phi_t)\dot \xi_{\delta}=0,\\\nonumber
 &\partial_t \Phi_t+\frac 12|\nabla \Phi_t|^2+\eta \frac 12|\nabla \Phi_t|^2\dot \xi_{\delta}=-\frac {\delta }{\delta \rho_t} \mathcal F(\rho_t)-\eta \frac {\delta }{\delta \rho_t} \Sigma(\rho_t)\dot \xi_{\delta},
 \end{align}
 where $(1+\dot \xi_{\delta}(t))\Phi_t=(-\Delta_{\rho_t})^{\dagger}\partial_t \rho_t$ up to a spatially constant stochastic process shift. 
\end{tm}

\begin{proof}
Consider the perturbations on $\rho$ and $\Phi$. 
Following the arguments in the proof of Proposition \ref{prop-sto}, the critical point satisfies that
\begin{align*}
&\mathcal S(\rho_t+\epsilon \delta \rho_t ,\Phi_t+\epsilon \delta \Phi_t)\\
&=\mathcal S(\rho_t,\Phi_t)
-\epsilon \int_0^T \<\Phi(t),  \partial_t \delta \rho_t\> dt 
-\epsilon \int_0^T \<\delta \Phi(t), \partial_t \rho_t\> dt\\
&+\epsilon \int_0^T \frac {\delta} { \delta \rho_t} \mathcal H_0(\rho_t,\Phi_t)\delta \rho_t +\frac {\delta} { \delta \Phi_t} \mathcal H_0(\rho_t,\Phi_t)\delta \Phi_t dt
\\
&+\epsilon \int_0^T \frac {\delta} { \delta \rho_t} \mathcal H_1(\rho_t,\Phi_t)\delta \rho_t +\frac {\delta} {\delta \Phi_t} \mathcal H_1(\rho_t,\Phi_t)\delta \Phi_t d\xi_{\delta}(t)+o(\epsilon)\\
&=\mathcal S(\rho_t,\Phi_t)
+\epsilon\int_0^T  \<\partial_t\Phi(t), \delta \rho_t\> dt 
-\epsilon\int_0^T \<\delta \Phi(t), \partial_t \rho_t\> dt\\
&+\epsilon\int_0^T \<\frac {\delta} { \delta \rho_t} \mathcal H_0(\rho_t,\Phi_t),\delta \rho_t \> +\< \frac {\delta} {\delta \Phi_t} \mathcal H_0(\rho_t,\Phi_t),\delta \Phi_t\> dt
\\
&+\epsilon\int_0^T \<\frac {\delta} {\delta \rho_t} \mathcal H_1(\rho_t,\Phi_t), \delta \rho_t\> +\<\frac {\delta} { \delta \Phi_t} \mathcal H_1(\rho_t,\Phi_t),\delta \Phi_t \> d\xi_{\delta}(t)+o(\epsilon).
\end{align*}
Taking $\epsilon \to 0$, we obtain that 
\begin{align*}
  \partial_t \rho_t &=  \frac {\delta} { \delta \Phi_t }\mathcal H_0(\rho_t,\Phi_t)+\frac {\delta } { \delta \Phi_t }\mathcal H_0(\rho_t,\Phi_t)\dot \xi_{\delta}(t)\\
  \partial_t \Phi_t &=  -\frac {\delta } { \delta \rho_t} \mathcal H_0(\rho_t,\Phi_t)-\frac {\delta } { \delta \rho_t} \mathcal H_0(\rho_t,\Phi_t)\dot \xi_{\delta}(t),
\end{align*}
which leads to \eqref{hs-rhos-mix}.
\hfill 

\end{proof}

Similarly, consider the action functional  
\begin{align*}
\mathcal S_{B}(\rho_t,\Phi_t)&=\<\rho(0),\Phi(0)\>-\<\rho(T),\Phi(T)\>+\int_0^T \<\rho_t, \circ d \Phi(t)\>+\mathcal H_0(\rho_t, \Phi_t) dt  \\&+\int_0^T \mathcal H_1(\rho_t,\Phi_t)\circ dB_t
\end{align*}
over the $\mathbb F_t$-adapted feasible set, we obtain the following stochastic system.

\begin{tm}\label{tm-sto-sto}
 The critical point of the variational problem $I(\rho^0,\rho^T)$ defined by  
   $$I (\rho^0,\rho^T)=\inf\{\mathcal S_B(\rho_t,\Phi_t)| \rho(0)=\rho^0, \rho(T)=\rho^T\}$$
 satisfies 
 the following  Hamiltonian system
 \begin{align}\label{hs-rhos-sto}
& \partial_t \rho_t+\nabla \cdot (\rho_t\nabla \Phi_t)+\eta \nabla \cdot (\rho_t\nabla \Phi_t)\circ dB_t=0,\\\nonumber
 &\partial_t \Phi_t+\frac 12|\nabla \Phi_t|^2+\eta \frac 12|\nabla \Phi_t|^2\circ dB_t=-\frac {\delta }{\delta \rho_t} \mathcal F(\rho_t)-\eta \frac {\delta }{\delta \rho_t} \Sigma(\rho_t) \circ dB_t
 \end{align}
 up to a spatially constant stochastic process shift on $\Phi_t$. 
\end{tm}

Next, we show that the continuity equation and the velocity equation generated by $\Phi$, 
\begin{align}\label{rho-v-wong}
& \partial_t \rho_t+\nabla \cdot (\rho_t v_t)+\eta \nabla \cdot (\rho_t v_t)\dot \xi_{\delta}=0,\\\nonumber
 &\partial_t v_t+\nabla v_t\cdot v_t+\eta \nabla v_t\cdot v_t \dot \xi_{\delta}=-\nabla \frac {\delta }{\delta \rho_t} \mathcal F(\rho_t)-\eta \frac {\delta }{\delta \rho_t} \nabla \Sigma(\rho_t)\dot \xi_{\delta}    
\end{align}
is convergent to the corresponding system driven by the Brownian motion. 

\begin{prop}
Assume that $v(0,\cdot),\rho(0,\cdot)$ is $\mathbb F_0$-measurable and  smooth, $\mathcal{F}(\rho_t)=\int_{\mathcal M} f \rho_t dvol_\mathcal M$ and $\Sigma(\rho_t)=\int_{\mathcal M} \sigma \rho_t dvol_\mathcal M$ with $f,\sigma \in C_p^3(\mathcal M).$
Let $\rho^{\delta},v^{\delta}$ be the solution of \eqref{rho-v-wong},  and $\rho,v$ be the solution of 
\begin{align}\label{rho-v-wong-sto}
& \partial_t \rho_t+\nabla \cdot (\rho_tv_t)+\eta \nabla \cdot (\rho_tv_t)\circ dB_t=0,\\\nonumber
 &\partial_t v_t+\nabla v_t \cdot v_t+\eta \nabla v_t \cdot v_t \circ dB_t=-\nabla \frac {\delta }{\delta \rho_t} \mathcal F(\rho_t)-\eta \nabla \frac {\delta }{\delta \rho_t} \Sigma(\rho_t)\circ dB_t.
 \end{align}
 Then there exists a stopping time $\tau>0$ 
 such that  
 \begin{align*}
 \lim_{\epsilon \to 0}\mathbb P (\sup_{t\in [0,\tau)}[|\rho_t^{\delta}-\rho_t|_{L^{\infty}(\mathcal M)}+|v_t^{\delta}-v_t|_{L^{\infty}(\mathcal M^d)}]> \epsilon)=0.
 \end{align*}
\end{prop}

\begin{proof}
Since $\mathcal M$ is compact, $f,\sigma \in C_p^3(\mathcal M)$, similar to the proofs of Lemma \ref{rd-won} and Lemma \ref{mani-wong}, we can obtain the global well-posedness of the particle ODE systems 
\begin{align*}
&dX_t=v(t,X_t)dt+\eta v(t,X_t)\circ dB_t,\\\nonumber
&dv(t,X_t)=-\nabla f(X_t)dt-\eta \nabla \sigma(X_t) \circ dB_t,
\end{align*}
and 
\begin{align*}
dX_t^{\delta}&=v^{\delta}(t,X_t^{\delta})dt+\eta v(t,X_t^{\delta}) d\xi_{\delta},\\\nonumber 
dv^{\delta }(t,X^{\delta}_t)&=-\nabla f(X_t^{\delta})dt-\eta \nabla \sigma(X_t^{\delta}) d\xi_\delta.
\end{align*}
 Following the arguments in the proof Proposition \ref{path-con-rhos}, we can obtain that there exists a stopping time $\tau>0$ such that $X_t$ is a smooth diffeomorphism before $\tau$. Notice that the density function $\rho^\delta(t,y)$ of $X^{\delta}_t$ satisfies $\rho^{\delta}(t,y)=|\det(\nabla X^{\delta}_t(y))|\rho(0,X^{\delta}_t(y)).$ Since $\rho(0,\cdot)$ is smooth for any fixed $\omega$ and the pathwise convergence of $X^{\delta}$ holds, it follows that $\rho^{\delta}(t,y)$ converges to 
the density function of $X_t$ before $\tau$, which is $\rho(t,y)=|\det( \nabla X_t(y))|\rho(0,X_t(y)).$
Similarly, the pathwise convergence of $v^{\delta}(t,X^{\delta}_t(y))$ to $v(t,X_t(y)),$
together with invertibility of $X^{\delta}_t$ and $X_t$,  implies the convergence of $v^{\delta}(t,x)$ to $v(t,x)$ before $\tau.$

\end{proof}

\begin{rk}
If one obtains the convergence of the Wong--Zakai approximations of the mean-field SODEs,
\begin{align*}
    &dX_t=v(t,X_t)dt+\eta v(t,X_t)\circ dB_t,\\\nonumber
&dv(t,X_t)=-\nabla \frac {\delta }{\delta \rho(t,X_t)} \mathcal{F}(\rho(t,X_t))dt-\eta \nabla \frac {\delta }{\delta \rho(t,X_t)} \Sigma (t,X_t) \circ dB_t,
\end{align*}
then the convergence of \eqref{rho-v-wong}  to \eqref{rho-v-wong-sto}  can be shown similarly before the stopping time $\tau$, that is, the first time $X_t$ is not a smooth diffeomorphis on $\mathcal M$ or $X_t$ escapes $\mathcal M$.
\end{rk}

\section{Examples}

In this section, we show that both the stochastic nonlinear Schr\"odinger (NLS) equation, which models the propagation of nonlinear dispersive waves in random or inhomogenous media in quantum physics (see e.g.  \cite{PhysRevE.49.4627,PhysRevE.63.025601,MR1425880,UEDA1992166}), and nonlinear Schr\"odinger equation with random dispersion, which describes the propagation of a signal in an optical fibre with dispersion management (see e.g. \cite{Agra01b,Agra01a}), are stochastic Wasserstein-Hamiltonian flows. We also discuss that the mean-field game system with common noise (see e.g. \cite{MR195142,MR0400425,MR3712946}) is a stochastic Wasserstein-Hamiltonian flow under suitable transformations.

\subsection{Stochastic NLS equation}
The dimensionless stochastic NLS equation is given by 
\begin{align}\label{snls}
d u=\bi \Delta u dt+\bi \lambda f(|u|^2) u dt+ \bi u \circ dW_t,
\end{align}
where $W_t$ is a Wiener process on the Hilbert space $L^2(\mathcal M; \mathbb R)$ and $f$ is a real-valued continuous function.
Since $Q$-Wiener process $W$ has the Karhunen--Lo$\grave{\text{e}}$ve expansion $W(t,x)=\sum_{i\in \mathbb N^+} Q^{\frac 12}e_i(x) \beta_i(t)$ {(see e.g. \cite{BD99})},  where $\{e_i\}_{i\in \mathbb N}$ is an orthonormal  basis of $L^2(\mathcal M; \mathbb R)$, and $\{\beta_i\}_{i\in \mathbb N}$ is a sequence of linearly independent Brownian motions on $(\Omega,\mathbb F,\{\mathbb F_t\}_{t\ge 0}, \mathbb P).$  We denote $W_{\delta}(t,x)=\sum_{i\in \mathbb N^+} Q^{\frac 12}e_i(x) \beta_i^{\delta}(t)$ as the piecewise linear Wong--Zakai approximation (or other type Wong--Zakai approximation) of $W$ and consider the approximated NLS equation of \eqref{snls}
\begin{align}\label{nls-wz}
\partial_t u(t,x)=\bi \Delta_{xx} u(t,x)+\bi \lambda f(|u(t,x)|^2) u(t,x)+ \bi u(t,x) \dot W_{\delta}(t,x).
\end{align}
We aim to prove that \eqref{nls-wz} is a stochastic Wasserstein Hamiltonian flow for any $\delta>0$, and thus its limit \eqref{snls} is also a stochastic Wasserstein Hamiltonian flow. In the following, we assume that $f$ is a real-value function, $W$ is smooth with respect to the space variable, and \eqref{nls-wz} possesses a mild solution or a strong solution on $[0,T].$

 Denote the $L^2$-inner product by $\<u,v\>=\Re \int_{\mathcal M} \bar uv dvol_M,$ where $\Re$ is the real part of a complex number. The variational problem on density manifold of \eqref{nls-wz} is 
 \begin{align}\label{var-snls-1}
   I_{\delta}(\rho^0,\rho^T)=\inf\{\mathcal S(\rho_t,\Phi_t)| (-\Delta_{\rho_t})^{\dagger}\Phi_t \in \mathcal T_{\rho_t} \mathcal P_{o}(\mathcal M),\rho(0)=\rho^0, \rho(T)=\rho^T\}   
 \end{align}
whose action functional is given by the dual coordinates, 
\begin{align*}
\mathcal S(\rho_t,\Phi_t)=-\int_0^T \< \Phi(t),\partial_t\rho_t\>dt +\int_0^T\mathcal H_0(\rho_t, \Phi_t) dt  + \sum_{i\in \mathbb N^+}\int_0^T \mathcal H_i(\rho_t,\Phi_t)d\beta_i^{\delta}(t).
\end{align*}
Here $\mathcal H_0(\rho_t, \Phi_t)=\int_{\mathcal M} |\nabla \Phi_t|^2\rho_t d vol_{\mathcal M}+\frac 14 I(\rho)+\mathcal  F(\rho_t)$, $\mathcal H_i(\rho_t,\Phi_t)=-\Sigma_i (\rho_t)=-\int_{\mathcal M} Q^{\frac 12}e_i \rho_t  dvol_{\mathcal M}$, $\mathcal F(\rho)=- \frac {\lambda}2 \int_{\mathcal M} \int_0^{\rho} f(s) ds  dvol_{\mathcal M}$ with a smooth function $f$, 
and $I(\rho)=\int_{\mathcal M}|\nabla \log(\rho)|^2\rho d\textrm{vol}_{\mathcal{M}}.$  



In the following, we show the relationship between the the variational problem \eqref{var-snls-1} and nonlinear Schr\"odinger equation with Wong--Zakai approximation \eqref{nls-wz} by using the Madelung transform \cite{Madelung27}.

\begin{prop}\label{equ-var-nls}
The critical point of the variational problem \eqref{var-snls-1} satisfies the  Madelung system of 
\eqref{nls-wz} on the support of $\rho_t.$  Conversely, the Madelung transform of \eqref{nls-wz} satisfies the critical point of \eqref{var-snls-1} on the support of $|u_t|.$
\end{prop}

\begin{proof}
By studying the perturbation on the dual coordinates, the arguments in the proof of Theorem \ref{tm-sto-mix} yield that the critical point of \eqref{var-snls-1} satisfies
 \begin{align*}
& \partial_t \rho_t+2\nabla \cdot (\rho_t\nabla \Phi_t)=0,\\\nonumber
 &\partial_t \Phi_t+|\nabla \Phi_t|^2=-1/4 \frac {\delta}{\delta \rho_t} I(\rho_t)-\frac {\delta }{\delta \rho_t} \mathcal F(\rho_t)- \dot W_{\delta}.
 \end{align*}
Define a complex valued function by $\widehat u(t,x)=\sqrt{\rho(t,x)}e^{\bi \Phi(t,x)}.$ One obtains the equation of $\widehat u(t,x)$ satisfying \eqref{nls-wz} on the support of $\rho_t$ by direct calculations. 

Conversely, using the Madelung transform of the solution $\sqrt{\rho(t,x)}e^{\bi S(t,x)}=u(t,x)$ where $\rho=|u|^2$ for \eqref{nls-wz}.   Then direct calculation leads to
\begin{align*}
&e^{\frac 12 \log(\rho)+\bi S}(\frac 12 \frac {\partial_t \rho}{\rho}+\bi  {\partial_t S})\\
&=\bi e^{\frac 12 \log(\rho)+\bi S} (\frac 12 \frac {\nabla \rho}{\rho}+\bi \nabla S)^2+\bi e^{\frac 12 \log(\rho)+\bi S}(\frac 12 \frac {\Delta \rho}{\rho}+\bi \Delta S-\frac 12 |\frac {\nabla \rho}{\rho}|^2)\\
&+\bi  e^{\frac 12 \log(\rho)+\bi S} (\lambda f(\rho)+ \dot W_{\delta })\\
&=\bi e^{\frac 12 \log(\rho)+\bi S}(\frac 14 (\frac {\nabla \rho}{\rho})^2-(\nabla S)^2+\bi  \frac {\nabla \rho}{\rho} \cdot \nabla S) +\bi e^{\frac 12 \log(\rho)+\bi S}(\frac 12 \frac {\Delta \rho}{\rho}+\bi \Delta S-\frac 12 |\frac {\nabla \rho}{\rho}|^2)\\
&+\bi  e^{\frac 12 \log(\rho)+\bi S} (\lambda f(\rho)+ \partial_t W_{\delta }).
\end{align*}
This implies that on the support or $|u_t|$, it holds that 
\begin{align}\label{rhos-law}
\partial_t \rho&= -2 \nabla \cdot (\rho \nabla S),\\\nonumber
\partial_t S
&=-|\nabla S|^2-\frac 14 \frac {\delta}{\delta \rho} I(\rho)+\lambda f(\rho)+\dot W_{\delta }.
\end{align}

\end{proof}

Based on the above result, taking spatial gradient on the potential $S$, we get the following system with the conservation law
\begin{align}\label{rhov-law}
\partial_t \rho&= - \nabla \cdot (\rho v),\\\nonumber
\partial_t v&= -\nabla_x v\cdot v-\nabla_x \frac 12 \frac {\delta}{\delta \rho} I(\rho)+2\lambda \nabla_x f(\rho)+2\nabla_x \dot  W_{\delta },
\end{align}
where $v(t,x)=2\nabla S(t,x).$

The following theorem indicates that the stochastic NLS equation is a stochastic Wasserstein Hamiltonian flow due to the convergence of the Wong--Zakai approximation. For convenience, let us assume that $\mathcal M=\mathbb T^d$ or $\mathbb R^d$ and consider the case that $W$ consists of a finite combinations of independent Brownian motions, i.e., $W(t,x)=\sum_{k=1}^N q_{k}(x)\beta_{k}(t),$ with $q_k(x)\in \mathbb H^{m}(\mathcal M)\cap W^{k,\infty}(\mathcal M)$ for some $m\in \mathbb N$ and $k\in \mathbb N^+$. Here $\mathbb H^{m}(\mathcal M), W^{k,\infty}(\mathcal M)$ are the standard Sobolev space.

\begin{tm}
Let $m\in \mathbb N$ and $k\in \mathbb N^+$. 
Suppose that the initial value of  \eqref{nls-wz} and \eqref{snls} 
$u_0\in  \mathbb H^{m}$ is $\mathbb F_0$-measurable and has any finite $p$-moment, $p\in \mathbb N^+$, and that  
$f$ is a real-valued continuous  function satisfies
\begin{align*}
\|f(|u|^2)u-f(|v|^2)v\|&\le L_f(R) \|u-v\|,\; \|u\|,\|v\|\le R,\\
\|f(|u|^2)u\|_{\mathbb H^1}&\le L_f(R)(1+\|u\|_{\mathbb H^1}),\; \|u\|_{\mathbb H^1}\le R,
\end{align*}
where $\lim_{R\to \infty}L_f(R)=\infty.$
The Wong--Zakai approximation \eqref{nls-wz} is convergent almost surely to 
the stochastic NLS equation \eqref{snls} up to a subsequence. 
\end{tm}

\begin{proof}
Since the driving noise is real-valued, the skew-symmetry of the NLS equation leads to the mass conservation laws for both \eqref{nls-wz} and \eqref{snls}. By the local Lipschitz property of $f(|\cdot|^2)(\cdot)$, one can obtain the existence of the unique mild solutions for both  \eqref{nls-wz} and \eqref{snls} in $\mathcal C([0,T],L^2)$ by a standard argument in \cite{BD99}. 
In order to study the converge in $L^2,$ let us define an approximation sequence $u_0^{R_1}\in \mathbb H^1, R_1 \to \infty$ of the initial value $u_0$, which can be taken by using  truncated Fourier series or spectral Galerkin method (see e.g. \cite{CHLZ19}). The growth condition of $f$ in $\mathbb H^1$ and the uniform boundedness assumption of $q_k$ lead to 
\begin{align*}
\E \Big [\sup_{t\in [0,T]}\|u_t^{R_1}\|^{2p}_{\mathbb H^1}\Big]\le C(T,R_1,p)<\infty,
\E \Big [\sup_{t\in [0,T]}\|u_t^{\delta,R_1}\|_{\mathbb H^1}^{2p}\Big] \le C(T,R_1,\delta,p)<\infty,
\end{align*}
where $p\ge 1$, $\lim_{R_1\to \infty} C(T,R_1,p)=\infty,$ $\lim_{R_1\to \infty} C(T,R_1,\delta,p)=\infty$.
Meanwhile, $u_t^{R_1}, u_t^{\delta,R_1}$ are convergent to $u_t, u_t^{\delta},$ a.s. in $\mathcal C([0,T];L^2)$ as $R_1\rightarrow \infty$, respectively up to a subsequnce. 
The continuity estimate of $u_t^{R_1}, u_t^{R_1,\delta}$,
\begin{align*}
    &\E \Big[\|u^{R_1}(t)-u^{R_1}(s)\|^{2p}\Big]\le C(T,R_1,p)|t-s|^{ p},\\
    &\E \Big[\|u^{R_1,\delta}(t)-u^{R_1,\delta}(s)\|^{2p}\Big]\le C(T,R_1,\delta,p) (|t-s|^{p}+|\delta|^p),
\end{align*}
can be obtained due to the mass conservation law and the continuity of $e^{\bi \Delta t}$. However, to get the convergence of \eqref{nls-wz}, we need a priori estimate of $u^{R_1,\delta}$ which is independent of $\delta.$ To this end, we study the enegry of the Wong--Zakai approximation, $H(u)=\int_{\mathcal M}\frac 12 |\nabla u|^2 dvol_{\mathcal M}- \frac \lambda 2 \int_{\mathcal M}\int_0^{|u|^2}f(s)ds  dvol_{\mathcal M},$  and obtain
\begin{align*}
H(u^{\delta}(t))=   H(u^{\delta}(0))+\int_0^t \<\nabla u^{\delta}(s), \bi u^{\delta}(s) \nabla dW^{\delta}(s)\>.
\end{align*}
By taking expectation, we get that  
\begin{align*}
  &\E \Big[\sup_{t\in [0,T]}H(u^{\delta}(t))\Big]\\
  &\le   \E \Big[H(u^{\delta}(0))\Big] +
  \E \Big[\sup_{t\in [0,T]}|\int_0^{[t]_{\delta}} \<\nabla u^{\delta}([s]_{\delta}), \bi u^{\delta}([s]_{\delta}) \nabla dW^{\delta}(s)\>|\Big]\\
  &+  \E \Big[\sup_{t\in [0,T]}|\int_{[t]_{\delta}}^t \<\nabla u^{\delta}([s]_{\delta}), \bi u^{\delta}([s]_{\delta}) \nabla dW^{\delta}(s)\>|\Big]\\
  &+ \E \Big[\sup_{t\in [0,T]}|\int_0^{[t]_{\delta}} \<\nabla u^{\delta}([s]_{\delta}), \bi (u^{\delta}(s)-u^{\delta}([s]_{\delta})) \nabla dW^{\delta}(s)\>|\Big]\\
  &+ \E \Big[\sup_{t\in [0,T]}|\int_{[t]_{\delta}}^t \<\nabla u^{\delta}([s]_{\delta}), \bi (u^{\delta}(s)-u^{\delta}([s]_{\delta})) \nabla dW^{\delta}(s)\>|\Big]\\
  &+ \E \Big [\sup_{t\in [0,T]} |\int_0^{[t]_{\delta}} \<\nabla (u^{\delta}(s)-u^{\delta}([s]_\delta)), \bi u^{\delta}(s) d W^{\delta}(s)\>|\Big] \\
   &+ \E \Big [\sup_{t\in [0,T]} |\int_{[t]_{\delta}}^{t} \<\nabla (u^{\delta}(s)-u^{\delta}([s]_\delta)), \bi u^{\delta}(s) d W^{\delta}(s)\>|\Big] \\
  &=\E \Big[H(u^{\delta}(0))\Big]+V_1+V_2+V_3+V_4+V_5+V_6. 
\end{align*} 
Below we show the estimates of $V_i$ ($i=1,\cdots,6$).
The Burkholder's inequality and mass conservation law lead to 
\begin{align*}
V_1 \le  \E \Big[\int_0^T C(H(u^{\delta}([t]_{\delta}))+C(\|u_0\|)) ds\Big].   
\end{align*}
Applying the Burkholder and Minkowski  inequalities, and the mass conservation law, we achieve that for $T=K\delta$,
\begin{align*}
V_2&\le  1+\E \Big[\sup_{t\in [0,T]}|\int_{[t]_{\delta}}^t \<\nabla u^{\delta}([s]_{\delta}), \bi u^{\delta}([s]_{\delta}) \nabla dW^{\delta}(s)\>|^2\Big]\\
&\le 1+ \sum_{k=0}^{K-1} \E \Big[\sup_{t\in [t_k,t_{k+1}]}|\int_{t_k}^t \<\nabla u^{\delta}(t_k), \bi u^{\delta}(t_k) \nabla dW(s)\>|^2\Big]\\
&\le 1+C \sum_{k=0}^{K-1} \E \Big [\sum_{i=1}^N \int_{t_k}^{t_{k+1}}\<\nabla u^{\delta}(t_k),\bi u^{\delta}(t_k)\nabla q_i(x)\>^2dt \Big]\\
&\le 1+C \sum_{i=1}^N \E \Big[ \|\nabla u^{\delta}([t]_{\delta})\|^2\|u^{\delta}([t]_{\delta})\|^2\|q_i\|^2_{W^{1,\infty}} dt\Big]\\
&\le1+ C \|u(0)\|^2\sum_{i=1}^N\|q_i\|^2_{W^{1,\infty}}  \int_{0}^T \E \Big[\|\nabla u^{\delta}([t]_{\delta})\|^2\Big] dt.
\end{align*}
The definition of $H$ leads to that there exists a constant $C(\|u_0\|)$ depending on $\|u_0\|$ such that  
\begin{align*}
    &\E \Big[\sup_{t\in [0,T]}|\int_{[t]_{\delta}}^t \<\nabla u^{\delta}([s]_{\delta}), \bi u^{\delta}([s]_{\delta}) \nabla dW^{\delta}(s)\>|^2\Big]\\
    &\le 2 C \|u_0\|^2\sum_{i=1}^N\|q_i\|^2_{W^{1,\infty}} \int_{0}^T \E \Big[  H(u^{\delta}([t]_{\delta}))\Big]dt
    + C(\|u_0\|).
\end{align*}
The mild form of $u^{\delta}(s)-u^{\delta}([s]_{\delta})$,
\begin{align*}
&u^{\delta}(s)-u^{\delta}([s]_{\delta})\\
&=(e^{\bi \Delta (s-[s]_{\delta})}-I)u^{\delta}([s]_{\delta})
+\int_{[s]_{\delta}}^s e^{\bi \Delta (s-r)}\bi \lambda f(|u^{\delta}(r)|^2)u^{\delta}(r) dr\\
&+\int_{[s]_{\delta}}^s\bi e^{\bi \Delta (s-r)} u^{\delta}(r) dW^{\delta}(r),
\end{align*}
together with the mass conservation law and $\|e^{\bi \Delta t}-I\|_{\mathcal L(\mathbb H^1,L^2)}\le C t^{\frac 12}$ (see, e.g., \cite{BD99}), yields that 
\begin{align}\label{con-hold}
\|u^{\delta}(s)-u^{\delta}([s]_{\delta})\|
&\le C \|u^{\delta}([s]_{\delta})\|_{\mathbb H^1}\delta^{\frac 12}
+ L_f(\|u_0\|)(1+\|u_0\|)\delta
\\\nonumber
&\quad+  C \|W([s]_{\delta}+\delta)-W([s]_{\delta})\|\|u_0\|. 
\end{align}
By making use of \eqref{con-hold} and the Burkholder's inequality, we obtain 
\begin{align*}
    V_3
&\le     
C (1+\E\Big [ \int_0^T \|\nabla u^{\delta }([s]_{\delta})\|^2 ds\Big])\\
&+ C(\|u_0\|)\E\Big [\int_0^T\|\nabla u^{\delta}([s]_{\delta})\|(1+\|u_0\|)\Big(\frac {\|W([s]_{\delta}+\delta)-W([s]_{\delta})\|_{L^{\infty}}^2}{\delta }\\
&+\|W([s]_{\delta}+\delta)-W([s]_{\delta})\|_{L^{\infty}}\Big)ds\Big]\\
&\le C(\|u_0\|)(1+\E\Big [ \int_0^T H (u^{\delta }([s]_{\delta})) ds\Big]).
\end{align*}
Similar arguments yield that
\begin{align*}
 V_4
 &\le C \E\Big [\sup_{t\in [0,T]} \int_{[t]_{\delta}}^{t} \|\nabla u^{\delta }([s]_{\delta})\|^2 \|W([s]_{\delta}+\delta)-W([s]_{\delta})\|\delta^{-\frac 12} ds\Big]\\
&+ C(\|u_0\|)\E\Big [\sup_{t\in [0,T]}\int_{[t]_{\delta}}^t \|\nabla u^{\delta}([s]_{\delta})\|(1+\|u_0\|)\Big(\frac {\|W([s]_{\delta}+\delta)-W([s]_{\delta})\|_{L^{\infty}}^2}{\delta }\\
&+\|W([s]_{\delta}+\delta)-W([s]_{\delta})\|_{L^{\infty}}\Big)ds\Big]\\
&\le C \delta \E\Big [\sup_{s\in [0,T]} H(u^{\delta}([s]_{\delta})) \Big]+ C(\|u_0\|)\delta.
\end{align*}
The estimates of $V_5$ and $V_6$ are omitted here since they are very similar to those of $V_3$ and $V_4$.
We conclude that
\begin{align*}
   & V_1+V_2+V_3+V_4+V_5+V_6\\
    &\le C\delta  \E \Big[\sup_{t\in [0,T]}H(u^{\delta}(t))\Big]+C\E \Big[\int_0^T (H(u^{\delta}([t]_{\delta}))dt\Big]
    +C(\|u_0\|).
\end{align*}
Thus, we obtain $\E \Big[\sup\limits_{t\in [0,T]}H(u^{\delta}(t))\Big]\le C(T,R_1,\|u_0\|)$ by using Gronwall's inequality and taking $\delta$ small enough. 
Similarly, it holds that for any $p\ge 1,$
\begin{align*}
&\E \Big[\sup_{t\in [0,T]}H^p(u^{\delta}(t))\Big]\le C(T,R_1,\|u_0\|,p),\\
&\E \Big[\|u^{R_1,\delta}(t)-u^{R_1,\delta}(s)\|^{2p}\Big]\le C(T,R_1,p) (|t-s|^{p}+|\delta|^p).
\end{align*}

Next, it suffices to prove the convergence of the Wong--Zakai approximation. To this end, we consider a stopping time $\tau=\inf\{t\in [0,T]| \|u^{R_1}(t)\|\ge R \; \text{or}\;  \|u^{\delta,R_1}([t]_{\delta})\|\ge R\}.$ In the following, we omit the supindex $R_1$. 
Applying the chain rule, we obtain that for $t\le \tau,$ 
\begin{align*}
&\|u(t)-u^{\delta}(t)\|^2=\|u(0)-u^{\delta}(0)\|^2+ 2\int_0^t\<\bi f(|u(s)|^2)u(s)-\bi f(|u^{\delta}(s)|^2)u^{\delta}(s), ~ u(s)-u^{\delta}(s)\>ds\\
& + 2\int_0^t \<u(s)-u^{\delta}(s),-\frac 12 \sum_{k=1}^N |q_k|^2 u(s) \>ds \\
&+2\int_0^t \<u(s)-u^{\delta}(s), \bi u(s) dW(s)-\bi u^{\delta}(s) dW_{\delta}(s)\>\\
&+\int_0^t \sum_{k=1}^N \int_{\mathcal M}|u(s)|^2|q_k|^2dvol_M ds\\
&\le \int_0^t 2 L_f(\|u(0)\|)  \|u(s)-u^{\delta}(s)\|^2 ds+  \int_0^t \<u^{\delta}(s),\sum_{k=1}^N |q_k|^2 u(s) \>ds\\
&-2\int_0^t \< u(s), \bi u^{\delta}(s)dW^{\delta}(s)\> -2\int_0^t\<u^{\delta}(s),\bi u(s)dW(s)\>\\
& \le \int_0^t 2 L_f(\|u(0)\|)  \|u(s)-u^{\delta}(s)\|^2 ds+  \int_0^t \<u^{\delta}(s),\sum_{k=1}^N |q_k|^2 u(s) \>ds\\
&-2\int_0^t \< u(s), \bi u^{\delta}([s]_{\delta})dW^{\delta}(s)\> 
-2\int_0^t \< u(s),\bi (u^{\delta}(s)-u^{\delta}([s]_{\delta}))dW^{\delta}(s) \> 
\\ &-2\int_0^t\<u^{\delta}([s]_{\delta}),\bi u(s)dW(s)\>-2\int_0^t\<u^{\delta}(s)-u^{\delta}([s]_{\delta}),\bi u(s)dW(s)\>\\
&=: 
\int_0^t 2 L_f(\|u(0)\|)  \|u(s)-u^{\delta}(s)\|^2 ds+III_1+III_2+III_3+III_4+III_5.
\end{align*}

For the term $III_2,$  the property of Wiener process, the mass conservation law, H\"older's and Young's inequality, as well as the property of the martingale,  yield that 
\begin{align*}
\E [III_2]&\le -2\int_{0}^{[t]_{\delta}} \E\Big[\<u(s)-u([s]_{\delta}),\bi u^{\delta}([s]_{\delta})dW^{\delta}(s)\>\Big]\\
&-2\int_{0}^{[t]_{\delta}} \E\Big[\<u([s]_{\delta}),\bi u^{\delta}([s]_{\delta})dW^{\delta}(s)\>\Big]+ C \delta^{\frac 12}\\
&\le C(1+C_{R})\delta^{\frac 12} -2\int_{0}^{[t]_{\delta}} \E\Big[\<\int_{[s]_{\delta}}^{s}\bi u([r]_{\delta}))dW(r),\bi u^{\delta}([s]_{\delta})dW^{\delta}(s)\>\Big]\\
&  -2\int_{0}^{[t]_{\delta}} \E\Big[\<\int_{[s]_{\delta}}^{s} (\exp(\bi \Delta (r-[s]_{\delta}))-I) \bi u([r]_{\delta}))dW(r),\bi u^{\delta}([s]_{\delta})dW^{\delta}(s)\>\Big] \\
&\le -2\int_{0}^{[t]_{\delta}} \E\Big[\<\int_{[s]_{\delta}}^{s}\bi u([r]_{\delta})dW(r),\bi u^{\delta}([s]_{\delta})dW^{\delta}(s)\>\Big]+ C(1+C_{R})\delta^{\frac 12}.
\end{align*}
Similar to $III_2,$ we have that
$\E [III_4] \le  C(1+C_R) \delta^{\frac  12}.$

For the terms $III_3$ and $III_5$, by taking expectation and using the property $\|e^{\bi \Delta t}-I\|_{\mathcal L(\mathbb H^1,L^2)}\le C t^{\frac 12}$, 
the continuity estimate of $u$ and the property of martingale, we arrive at 
\begin{align*}
\E\Big[ III_3\Big]
&\le  -\int_0^{[t]_{\delta}} 2\E\Big [\< u(s)-u([s]_{\delta}), \bi (u^{\delta}(s)-u^{\delta}([s]_{\delta}) )dW^{\delta}(s)\>\Big ]\\
& -\int_0^{[t]_{\delta}} 2\E\Big [\<u([s]_{\delta}), \bi (u^{\delta}(s)-u^{\delta}([s]_{\delta}) )dW^{\delta}(s)\>\Big ] +C(1+C_{R})\delta^{\frac 12}.\\
&= -\int_0^{[t]_{\delta}} 2\E\Big [\<u([s]_{\delta}), \bi \left( \int_{[s]_{\delta}}^{s} \bi u^\delta ([r]_{\delta}) dW^{\delta}(r) \right)  dW^{\delta}(s)\>\Big]  +C(1+C_{R})\delta^{\frac 12},\\
\E [III_5]&\le  -2\E \Big [\int_0^{[t]_{\delta}}\<\int_{[s]_{\delta}}^s \bi u^{\delta}([r]_{\delta}) dW^{\delta}(r),\bi u([s]_{\delta})dW(s)\>\Big ] +C(1+C_{R})\delta^{\frac 12}.
\end{align*}

Due to the independent increments of $W$ and the property of conditional expectation, we obtain that   
\begin{align*}
&2\int_{0}^{[t]_{\delta}} \E\Big[\<\int_{[s]_{\delta}}^{s}\bi u([r]_{\delta}))dW(r),\bi u^{\delta}([s]_{\delta})dW^{\delta}(s)\>\Big]\\
&= 
2\sum_{k=0}^{\frac {[t]_{\delta}}\delta -1} \E \Big [ \int_{t_k}^{t_{k+1}}\<u(t_k)(W(s)-W(t_k)),u^{\delta}(t_k)(W(t_{k+1})-W(t_k))\>\delta^{-1} \Big] ds 
\\
&=2\sum_{k=0}^{\frac {[t]_{\delta}}\delta -1} \E \Big [ \int_{t_k}^{t_{k+1}}\frac {s-t_k} {\delta }\sum_{i=1}^N \<u(t_k),u^{\delta}(t_k)|q_i|^2\> \Big] ds \\
&=\int_{0}^{[t]_{\delta}} \E \Big[\<u^{\delta}([s]_{\delta}),\sum_{k=1}^N|q_k|^2u([s]_{\delta})\>\Big]ds.
\end{align*}
On the other hand, $\int_{[t]_{\delta}}^t \E \Big[\<u^{\delta}([s]_{\delta}),\sum_{i=1}^N|q_i|^2u([s]_{\delta})\>\Big]ds\le C \delta$ due to the mass conservation law and assumption on $q_i$.

Combining the above estimates, we obtain that 
\begin{align*}
 &\E\Big[\|u(t)-u^{\delta}(t)\|^2\Big]\\
 &\le \int_0^t 2L_f(R)\E\Big[\|u(s)-u^{\delta}(s)\|^2\Big] 
 +C(1+C_{R})\delta^{\frac 12} 
 +\int_0^t \E \Big[\<u^{\delta}(s),\sum_{i=1}^N|q_i|^2u(s)\>\Big]ds\\
 &-2\int_{0}^{[t]_{\delta}} \E\Big[\<\int_{[s]_{\delta}}^{s}\bi u([r]_{\delta}))dW(r),\bi u^{\delta}([s]_{\delta})dW^{\delta}(s)\>\Big]\\
 &\le 
 \int_0^t 2L_f(\|u(0)\|)\E\Big[\|u(s)-u^{\delta}(s)\|^2\Big] 
 +C(1+C_{R})\delta^{\frac 12} 
 +\int_0^t \E \Big[\<u^{\delta}(s),\sum_{i=1}^N|q_i|^2u(s)\>\Big]ds\\
 &-\int_{0}^{[t]_{\delta}} \E \Big[\<u^{\delta}([s]_{\delta}),\sum_{i=1}^N|q_i|^2u([s]_{\delta})\>\Big]ds.
\end{align*}

Applying the Gronwall's inequality and the continuity estimate of $u$ and $u^{\delta}$, we get 
\begin{align*}
\E[\|u(t)-u^{\delta}(t)\|^2]&\le C(1+C_{R})\exp(2L_f(\|u(0)\|)T)\delta^{\frac 12}.
\end{align*}
It follows that 
\begin{align*}
   &\mathbb P (\|u(t)-u^{\delta}(t)\|> \epsilon ) \\
  &\le 
  \mathbb P (\|u^{R_1}(t)-u(t)\|> \frac \epsilon 3)+  \mathbb P (\|u^{R_1,\delta}(t)-u^{\delta}(t)\|> \frac \epsilon 3 )\\
  &+ \mathbb P(\|u^{R_1}(t)-u^{R_1,\delta}(t)\|> \frac \epsilon 3, t\le \tau)  +\mathbb P(\|u^{R_1}(t)-u^{R_1,\delta}(t)\|> \frac \epsilon 3, t>\tau).
\end{align*}
Taking limit on $\delta\to 0$, $R,R_1 \to \infty$, using the strong convergence estimate and Chebyshev's inequality,  we obtain 
\begin{align*}
&\lim_{\delta \to 0}\mathbb P (\|u(t)-u^{\delta}(t)\|> \epsilon ) \\
&\le 
 \lim_{\delta \to 0}\frac {9}{\epsilon^2}C(1+C_{R})\exp(2L_f(\|u_0\|)T)\delta^{\frac 12} \\
 &
+ \lim_{R\to \infty}\mathbb P (\sup_{s\in [0,t]}\|u(s)\|\ge R)+\lim_{R\to \infty} \mathbb P (\sup_{s\in [0,t]}\|u^{\delta}([s]_{\delta})\|\ge R)
=0.
\end{align*}
Similarly, following the above arguments, we further obtain 
\begin{align*}
\lim_{\delta \to 0}\E[\sup_{t\in [0,T]} \|u(t)-u^{\delta}(t)\|^2]&=0,
\end{align*}
which implies that 
\begin{align*}
  &\lim_{\delta \to 0}\mathbb P (\sup_{t\in [0,T]}\|u(t)-u^{\delta}(t)\|> \epsilon)=0.
\end{align*}
    
\end{proof}

\begin{rk}
Similar to the stochastic Wasserstein Hamiltonian flow induced by classical Stochastic ODEs, one may expect the particle version of the stochastic nonlinear Schr\"odinger equation \eqref{snls}, that is,
\begin{align}\label{MV-sodes}
&dX_t=v(t,X_t),\;\\\nonumber
& d v(t,X_t)
=-\nabla_{X_t} \frac 12 \frac {\delta}{\delta \rho} I(\rho(t,X_t))+2\lambda \nabla_{X_t} f(\rho(t,X_t)) +2\nabla_{X_t} \circ dW(t).
\end{align}
But we have not found a rigorous way to prove it. This will be studied in the future. 
\end{rk}

\subsection{NLS equation with random dispersion}
The dimensionless NLS equation with random dispersion is given by 
\begin{align}\label{snls-rds}
d u=\bi \Delta u \frac 1{\epsilon}m(\frac t{\epsilon^2})dt+\bi \lambda f(|u|^2) u dt,
\end{align}
where $m$ is a real-valued centered stationary random process. Under ergodic assumptions on $m$, it is expected that the limiting model when $\epsilon\to 0$ is the following stochastic NLS equation with white noise dispersion
\begin{align}\label{snls-dis}
du=\sigma_0 \bi \Delta u\circ dB_t +\bi\lambda f(|u|^2) u dt,
\end{align} 
where $\sigma_0^2=2\int_0^{\infty}\E[m(0)m(t)]dt$ (see e.g. \cite{BD10}). For simplicity, we set $\sigma_0=1$ in \eqref{snls-dis} throughout this section. 

To see \eqref{snls-dis} as a stochastic Wasserstein Hamiltonian flow, 
let us use \eqref{snls-rds} instead of Wong--Zakai approximations. Assume that the real valued centered stationary process $m(t)$ is continuous and such that for any $T>0$, $t\mapsto \epsilon\int_0^{\frac t {\epsilon^2}}m(s)ds$ converges in distribution to a standard real-valued Brownian motion $B$ in $\mathcal C([0,T])$ (see e.g. \cite{BD10}).

First, using Madelung transform $u(t,x)=\sqrt{\rho(t,x)}e^{\bi S(t,x)}$ gives 
\begin{align*}
&e^{\frac 12 \log(\rho)+\bi S}(\frac 12 \frac {\partial_t \rho}{\rho}+\bi  {\partial_t S})\\
&=\bi e^{\frac 12 \log(\rho)+\bi S} \Big(\frac 12 \frac {\nabla \rho}{\rho}+\bi \nabla S)^2+(\frac 12 \frac {\Delta \rho}{\rho}+\bi \Delta S-\frac 12 |\frac {\nabla \rho}{\rho}|^2) \Big)\frac 1{\epsilon}m(\frac t{\epsilon^2})\\
&+\bi  e^{\frac 12 \log(\rho)+\bi S} \lambda f(\rho)\\
&=\bi e^{\frac 12 \log(\rho)+\bi S}\Big(\frac 14 (\frac {\nabla \rho}{\rho})^2-(\nabla S)^2+\bi  \frac {\nabla \rho}{\rho} \cdot \nabla S) +(\frac 12 \frac {\Delta \rho}{\rho}+\bi \Delta S-\frac 12 |\frac {\nabla \rho}{\rho}|^2)\Big)\frac 1{\epsilon}m(\frac t{\epsilon^2})\\
&+\bi  e^{\frac 12 \log(\rho)+\bi S} \lambda f(\rho).
\end{align*}
We obtain that 
\begin{align}\label{dis-wz}
\partial_t \rho&= -2 \nabla \cdot (\rho \nabla S)\frac 1{\epsilon}m(\frac t{\epsilon^2}),\\\nonumber
\partial_t S
&=(-|\nabla S|^2-\frac 14 \frac {\delta}{\delta \rho} I(\rho))\frac 1{\epsilon}m(\frac t{\epsilon^2})+\lambda f(\rho),
\end{align}
which can be rewritten as
\begin{align*}
\partial_t \rho&= - \nabla \cdot (\rho v) \frac 1{\epsilon}m(\frac t{\epsilon^2}),\\\nonumber
\partial_t v&= (-\nabla_x v\cdot v-\nabla_x \frac 12 \frac {\delta}{\delta \rho} I(\rho))\frac 1{\epsilon}m(\frac t{\epsilon^2})+2\lambda \nabla_x f(\rho).
\end{align*}

Based on the above calculations, following the similar steps in the proof of Proposition \ref{equ-var-nls}, we conclude the following result. 

\begin{prop}
The critical point of the variational problem 
\begin{align}\label{var-snls-2}
   I_{\epsilon}(\rho^0,\rho^T)=\inf\{\mathcal S(\rho_t,\Phi_t)| (-\Delta_{\rho_t})^{\dagger}\Phi_t \in \mathcal T_{\rho_t} \mathcal P_{o}(\mathcal M),\rho(0)=\rho^0, \rho(T)=\rho^T\}   
 \end{align}
whose action functional is given by the dual coordinates, 
\begin{align*}
\mathcal S(\rho_t,\Phi_t)&=-\int_0^T \<\Phi(t),\partial_t \rho_t\>dt +\int_0^T \mathcal H_0(\rho_t, \Phi_t) dt 
+\int_0^T \mathcal H_1(\rho_t,\Phi_t)\frac 1{\epsilon}m(\frac t{\epsilon^2})dt,
\end{align*}
satisfies \eqref{dis-wz}.  
Here $\mathcal H_0(\rho_t, \Phi_t)=-\lambda \int_{\mathcal M}\int_0^{\rho} f(s)ds dvol_{\mathcal M}$ with a smooth function $f$, $\mathcal H_1(\rho_t,\Phi_t)=\int_{\mathcal M} |\nabla \Phi_t|^2\rho_t d vol_{\mathcal M}+\frac 14 I(\rho)$, $I(\rho)=\int_{\mathcal M}|\nabla \log(\rho)|^2\rho d vol_{\mathcal M}.$
\end{prop}

It has been shown in \cite{BD10} that the limit of \eqref{dis-wz} is the NLS equation with white noise dispersion. Therefore, \eqref{dis-wz} is also a stochastic Wasserstein Hamiltonian flow on density manifold.

\begin{rk}
The above system is also expected to have a particle version. By applying the push-forward map in section \ref{swhf} on $\widetilde \Omega$,  the particle version of \eqref{snls-dis} is expected to be 
\begin{align*}
&d{X_t}=v(t,X_t) \circ dB_t\;\\\nonumber
&dv(t,X_t)=-\nabla_{X_t} \frac 12 \frac {\delta}{\delta \rho} I(\rho(t,X_t))  \circ dB_t+2\lambda \nabla_{X_t} f(\rho(t,X_t)).
\end{align*}
And this will be studied in the future. 
\end{rk}

\subsection{Schr\"odinger Bridge Problem (SBP) with common noise}

In this part, we indicates that the critical point of the Schr\"odinger bridge problem (SBP) with common noise may  also be a stochastic Wasserstein Hamiltonian flow.   The SBP with common noise is inspired by 
\cite{MR3967062,wu2019viscosity} for the Schr\"{o}dinger Bridge type problem in stochastic case, where the common noise is added into the classical Schr\"{o}dinger Bridge type problem \cite{leonard2013survey,chen2020stochastic}. This problem can be  formulated as a stochastic control problem on Wasserstein manifold:
\begin{align}
   & \min_{\{v_t\}_{t\in [0,T]}}   \left[\int_0^T\int_{\mathbb{R}^d} \frac{1}{2}|v_t(x)|^2\rho_t(x,\omega)~dx~dt \right] \label{stochastic control on wasserstein mfld}\\
   & \textrm{Subject to:} ~ ~ \frac{\partial\rho_t(x,\omega)}{\partial t} +\nabla\cdot(\rho_t(x,\omega)(v_t+A(x,t)\dot{W}_t(\omega))) = \Delta \rho_t.\label{costraint of stochastic SBP}\\
   & \quad \quad \textrm{and}~~ \rho_0(\cdot,\omega)=\rho_a,~ \rho_T(\cdot,\omega) = \rho_b. \label{boundary condition of stochastic SBP}
\end{align} 
The continuity equation \eqref{costraint of stochastic SBP} can be viewed as an SDE on the Wasserstein manifold $\mathcal{P}_2(\mathbb{R}^d)$, which reads 
\begin{align*}
    dX_t=v(t,X_t)dt+\sqrt{2}dB(t)+A(t,X_t)dW(t). 
\end{align*}
Here $B$ is the Brownian motion which corresponding to the diffusion effect in \eqref{costraint of stochastic SBP}, and $W$ is another Brownian motion which is independent of $B$ and is called the common noise.

In the following, we consider the Wong--Zakai approximation of \eqref{stochastic control on wasserstein mfld}, i.e,
\begin{align}\label{WZ-SCP}
& \min_{\{v_t\}_{t\in [0,T]}}  \left[\int_0^T\int_{\mathbb{R}^d} \frac{1}{2}|v_t(x)|^2\rho_t(x,\omega)~dx~dt \right]\\\nonumber 
   & \textrm{Subject to:} ~ ~ \frac{\partial\rho_t(x,\omega)}{\partial t} +\nabla\cdot(\rho_t(x,\omega)(v_t+A(x,t)\dot \xi_{\delta}(t)) = \Delta \rho_t.\\\nonumber 
   & \quad \quad \textrm{and}~~ \rho_0(\cdot,\omega)=\rho_a,~ \rho_T(\cdot,\omega) = \rho_b, \nonumber 
\end{align} 
and show that its critical point is a Wasserstein Hamiltonian flow.

\begin{prop}\label{prop-sbp}
Assume that $W$ is $d$-dimensional Brownian motion, $\xi$ is the piecewisely linear Wong--Zakai approximation of $W.$ Let $A(\cdot, t)\in \mathcal C_b^1(\mathbb{R}^d), \rho_a,\rho_b\in \mathcal P_o(\mathbb{R}^d)$ be smooth. Then the critical point of \eqref{WZ-SCP} satisfies 
\begin{align}\label{sto-sbp}
&\partial_t \rho_t=\frac {\delta}{\delta \Phi} \mathcal H_0(\rho_t,\Phi_t)+\sum_{i=1}^d \frac {\delta}{\delta \Phi} \mathcal H_i(\rho_t,\Phi_t)(\dot \xi_{\delta})_i (t),\\\nonumber 
& \partial_t \Phi_t=-\frac {\delta}{\delta \rho} \mathcal H_0(\rho_t,\Phi_t)-\sum_{i=1}^d \frac {\delta}{\delta \rho} \mathcal H_i(\rho_t,\Phi_t)(\dot \xi_{\delta})_i (t),
\end{align}
where $\mathcal H_0(\rho,\Phi)=\frac 12 \int_{\mathcal M} |\nabla \Phi|^2\rho dvol_{\mathcal M}-\frac 18 I(\rho),$ $\mathcal H_i(\rho,\Phi)= \int_{\mathcal M}\rho A^i_t \partial_{x_i} \Phi dvol_{\mathcal M}.$ Here $A^i_t$ denotes the $i$-th column of the matrix $A_t$.

\end{prop}

\begin{proof}
By using the Lagrangian multiplier method, the critical point satisfies 
\begin{align}\label{f-pde}
&\partial_t \rho_t+\nabla\cdot(\rho(\nabla S_t+ A_t\dot\xi_{\delta}(t)))=\frac 12 \Delta \rho_t,
\\\label{b-pde}
& \partial_t S_t+\frac 12|\nabla S_t|^2+\nabla S_t \cdot A_t\dot \xi_{\delta}(t)=-\frac 12 \Delta S_t.
\end{align}
Applying the ``Hopf-Cole" transform (see e.g. \cite{leger2019hopfcole}) $\Phi_t=S_t-\frac 12 \log(\rho_t)$,  we obtain 
\begin{align*}
&\partial_t \rho_t +\nabla\cdot(\rho_t\nabla \Phi_t)+\nabla\cdot(\rho_t A_t \dot \xi_{\delta}(t))=0,\\
& \partial_t \Phi_t+\frac 12 |\nabla \Phi_t|^2+\nabla \Phi \cdot A_t \dot \xi_{\delta}(t)=\frac 18 \frac {\delta}{\delta \rho} I(\rho),
\end{align*}
which implies \eqref{sto-sbp}. 
\end{proof}

The above result also coincides with the generalized variational principle \eqref{gen-var-pri} with the action functional
\begin{align*}
\mathcal S(\rho_t,\Phi_t)&=-\int_0^T \< \Phi(t),\partial_t \rho_t\>dt+\int_0^T\mathcal H_0(\rho_t, \Phi_t) dt  +\sum_{i=1}^d\int_0^T \mathcal H_i(\rho_t,\Phi_t)d\xi_{\delta}(t),
\end{align*} 
whose critical point is 
the stochastic Hamiltonian system \eqref{sto-sbp}.
From the  proof of Proposition \ref{prop-sbp},  \eqref{sto-sbp} is equivalent to the forward and backward system which contains
the backward stochastic Hamilton-Jacobi equation \eqref{b-pde} and  a forward stochastic Kolmogorov equation \eqref{f-pde}, and plays the role of characteristics for the master equation \cite{MR3967062}.
The derivation of \eqref{sto-sbp}  may be extended to the mean-field game systems with common noise in \cite{MR3967062,MR3572323} up to an It\^o-Wentzell correction term \cite{MR2800911}. If the Wong--Zakai approximation \eqref{WZ-SCP} is convergent to \eqref{stochastic control on wasserstein mfld}, then the critical point of \eqref{stochastic control on wasserstein mfld} is expected to be a stochastic Wasserstein Hamiltonian flow. This will be our future research.

\section{Conclusions}
In this paper, we study the stochastic Wasserstein Hamiltonian flows, including the stochastic Euler--Lagrange equations and its Hamiltonian flows on density manifold. First, we show that the classical Hamiltonian motions with random perturbations and random initial data induce the stochastic Wasserstein Hamiltonian flows via Wong--Zakai approximation with Lagrangian formalism. Then we propose a generalized variational principle to derive and investigate the generalized stochastic Wasserstein Hamiltonian flows, including the stochastic nonlinear Schr\"odinger equation, Schr\"odinger equation with random dispersion and stochastic Schr\"odinger bridge problem. The study provides rigorous mathematical justification for the principle that the conditional probability density of stochastic Hamiltonian flow in sample space is stochastic Hamiltonian flow on density manifold.

\bibliographystyle{plain}
\bibliography{bib}

\begin{thebibliography}{10}

\bibitem{Agra01b}
G.~P. Agrawal.
\newblock {\em Applications of {N}onlinear {F}iber {O}ptics}.
\newblock Academic Press, San Diego, 2001.

\bibitem{Agra01a}
G.~P. Agrawal.
\newblock {\em Nonlinear {F}iber {O}ptics, 3rd ed.}
\newblock Academic Press, San Diego, 2001.

\bibitem{MR2361303}
L.~Ambrosio and W.~Gangbo.
\newblock Hamiltonian {ODE}s in the {W}asserstein space of probability
  measures.
\newblock {\em Comm. Pure Appl. Math.}, 61(1):18--53, 2008.

\bibitem{PhysRevE.49.4627}
O.~Bang, P.~L. Christiansen, F.~If, K.~\O{}. Rasmussen, and Y.~B. Gaididei.
\newblock Temperature effects in a nonlinear model of monolayer scheibe
  aggregates.
\newblock {\em Phys. Rev. E}, 49:4627--4636, May 1994.

\bibitem{MR1313027}
Z.~Brze\'{z}niak and F.~Flandoli.
\newblock Almost sure approximation of {W}ong-{Z}akai type for stochastic
  partial differential equations.
\newblock {\em Stochastic Process. Appl.}, 55(2):329--358, 1995.

\bibitem{MR3967062}
P.~Cardaliaguet, F.~Delarue, J.~M. Lasry, and P.~L. Lions.
\newblock {\em The master equation and the convergence problem in mean field
  games}, volume 201 of {\em Annals of Mathematics Studies}.
\newblock Princeton University Press, Princeton, NJ, 2019.

\bibitem{MR3753660}
R.~Carmona and F.~Delarue.
\newblock {\em Probabilistic theory of mean field games with applications.
  {II}}, volume~84 of {\em Probability Theory and Stochastic Modelling}.
\newblock Springer, Cham, 2018.
\newblock Mean field games with common noise and master equations.

\bibitem{MR3572323}
R.~Carmona, F.~Delarue, and D.~Lacker.
\newblock Mean field games with common noise.
\newblock {\em Ann. Probab.}, 44(6):3740--3803, 2016.

\bibitem{chen2020stochastic}
Y.~Chen, T.~T. Georgiou, and M.~Pavon.
\newblock On the relation between optimal transport and {S}chr\"{o}dinger
  bridges: a stochastic control viewpoint.
\newblock {\em J. Optim. Theory Appl.}, 169(2):671--691, 2016.

\bibitem{CHLZ12}
S.~Chow, W.~Huang, Y.~Li, and H.~Zhou.
\newblock Fokker-{P}lanck equations for a free energy functional or {M}arkov
  process on a graph.
\newblock {\em Arch. Ration. Mech. Anal.}, 203(3):969--1008, 2012.

\bibitem{CLZ19}
S.~Chow, W.~Li, and H.~Zhou.
\newblock A discrete {S}chr\"{o}dinger equation via optimal transport on
  graphs.
\newblock {\em J. Funct. Anal.}, 276(8):2440--2469, 2019.

\bibitem{CLZ20}
S.~Chow, W.~Li, and H.~Zhou.
\newblock Wasserstein {H}amiltonian flows.
\newblock {\em J. Differential Equations}, 268(3):1205--1219, 2020.

\bibitem{CP17}
C.~Conforti and M.~Pavon.
\newblock Extremal flows on {W}asserstein space.
\newblock {\em arXiv:1712.02257}, 2017.

\bibitem{CLZ20a}
J.~Cui, L.~Dieci, and H.~Zhou.
\newblock Time discretizations of {W}asserstein-{H}amiltonian flows.
\newblock {\em arXiv:2006.09187}, 2020.

\bibitem{CHLZ19}
J.~Cui, J.~Hong, and Z.~Liu.
\newblock Strong convergence rate of finite difference approximations for
  stochastic cubic {S}chr\"{o}dinger equations.
\newblock {\em J. Differential Equations}, 263(7):3687--3713, 2017.

\bibitem{CLZ21}
J.~Cui, S.~Liu, and H.~Zhou.
\newblock What is a stochastic {H}amiltonian process on finite graph? {A}n
  optimal transport answer.
\newblock {\em J. Differential Equations}, 305:428--457, 2021.

\bibitem{BD99}
A.~de~Bouard and A.~Debussche.
\newblock A stochastic nonlinear {S}chr\"odinger equation with multiplicative
  noise.
\newblock {\em Comm. Math. Phys.}, 205(1):161--181, 1999.

\bibitem{BD10}
A.~de~Bouard and A.~Debussche.
\newblock The nonlinear {S}chr\"{o}dinger equation with white noise dispersion.
\newblock {\em J. Funct. Anal.}, 259(5):1300--1321, 2010.

\bibitem{PhysRevE.63.025601}
G.~E. Falkovich, I.~Kolokolov, V.~Lebedev, and S.~K. Turitsyn.
\newblock Statistics of soliton-bearing systems with additive noise.
\newblock {\em Phys. Rev. E}, 63:025601, Jan 2001.

\bibitem{ferniq}
X.~Fernique.
\newblock Int\'{e}grabilit\'{e} des vecteurs gaussiens.
\newblock {\em C. R. Acad. Sci. Paris S\'{e}r. A-B}, 270:A1698--A1699, 1970.

\bibitem{MR1368671}
M.~Furi.
\newblock Second order differential equations on manifolds and forced
  oscillations.
\newblock In {\em Topological methods in differential equations and inclusions
  ({M}ontreal, {PQ}, 1994)}, volume 472 of {\em NATO Adv. Sci. Inst. Ser. C
  Math. Phys. Sci.}, pages 89--127. Kluwer Acad. Publ., Dordrecht, 1995.

\bibitem{MR2808856}
W.~Gangbo, H.~Kim, and T.~Pacini.
\newblock Differential forms on {W}asserstein space and infinite-dimensional
  {H}amiltonian systems.
\newblock {\em Mem. Amer. Math. Soc.}, 211(993):vi+77, 2011.

\bibitem{MR3195844}
D.~A. Gomes and J.~Sa\'{u}de.
\newblock Mean field games models---a brief survey.
\newblock {\em Dyn. Games Appl.}, 4(2):110--154, 2014.

\bibitem{MR0423094}
E.~Hille and R.~S. Phillips.
\newblock {\em Functional analysis and semi-groups}.
\newblock American Mathematical Society Colloquium Publications, Vol. XXXI.
  American Mathematical Society, Providence, R. I., 1974.
\newblock Third printing of the revised edition of 1957.

\bibitem{Hsu02}
E.~P. Hsu.
\newblock {\em Stochastic analysis on manifolds}, volume~38 of {\em Graduate
  Studies in Mathematics}.
\newblock American Mathematical Society, Providence, RI, 2002.

\bibitem{MR637061}
N.~Ikeda and S.~Watanabe.
\newblock {\em Stochastic differential equations and diffusion processes},
  volume~24 of {\em North-Holland Mathematical Library}.
\newblock North-Holland Publishing Co., Amsterdam-New York; Kodansha, Ltd.,
  Tokyo, 1981.

\bibitem{burkholder1972integral}
I.~Karatzas and S.~E. Shreve.
\newblock {\em Brownian motion and stochastic calculus}, volume 113 of {\em
  Graduate Texts in Mathematics}.
\newblock Springer-Verlag, New York, second edition, 1991.

\bibitem{MR1214374}
P.~E. Kloeden and E.~Platen.
\newblock {\em Numerical solution of stochastic differential equations},
  volume~23 of {\em Applications of Mathematics (New York)}.
\newblock Springer-Verlag, Berlin, 1992.

\bibitem{MR1425880}
V.~V. Konotop and L.~V\'{a}zquez.
\newblock {\em Nonlinear random waves}.
\newblock World Scientific Publishing Co., Inc., River Edge, NJ, 1994.

\bibitem{MR2800911}
N.~V. Krylov.
\newblock On the {I}t\^{o}-{W}entzell formula for distribution-valued processes
  and related topics.
\newblock {\em Probab. Theory Related Fields}, 150(1-2):295--319, 2011.

\bibitem{MR924776}
J.~D. Lafferty.
\newblock The density manifold and configuration space quantization.
\newblock {\em Trans. Amer. Math. Soc.}, 305(2):699--741, 1988.

\bibitem{leger2019hopfcole}
F.~L\'{e}ger and W.~Li.
\newblock Hopf-{C}ole transformation via generalized {S}chr\"{o}dinger bridge
  problem.
\newblock {\em J. Differential Equations}, 274:788--827, 2021.

\bibitem{leonard2013survey}
Christian L\'{e}onard.
\newblock A survey of the {S}chr\"{o}dinger problem and some of its connections
  with optimal transport.
\newblock {\em Discrete Contin. Dyn. Syst.}, 34(4):1533--1574, 2014.

\bibitem{Madelung27}
E.~Madelung.
\newblock Quanten theorie in hydrodynamischer form.
\newblock {\em Zeitschrift für Physik}, 40(3-4):322--326, 1927.
\newblock cited By 1026.

\bibitem{Nelson19661079}
E.~Nelson.
\newblock Derivation of the {S}chr\"odinger equation from {N}ewtonian
  mechanics.
\newblock {\em Physical Review}, 150(4):1079--1085, 1966.

\bibitem{MR0343816}
E.~Nelson.
\newblock The free {M}arkoff field.
\newblock {\em J. Functional Analysis}, 12:211--227, 1973.

\bibitem{MR783254}
E.~Nelson.
\newblock {\em Quantum fluctuations}.
\newblock Princeton Series in Physics. Princeton University Press, Princeton,
  NJ, 1985.

\bibitem{MR870196}
E.~Nelson.
\newblock Field theory and the future of stochastic mechanics.
\newblock In {\em Stochastic processes in classical and quantum systems
  ({A}scona, 1985)}, volume 262 of {\em Lecture Notes in Phys.}, pages
  438--469. Springer, Berlin, 1986.

\bibitem{MR0400425}
D.~W. Stroock and S.~R.~S. Varadhan.
\newblock On the support of diffusion processes with applications to the strong
  maximum principle.
\newblock In {\em Proceedings of the {S}ixth {B}erkeley {S}ymposium on
  {M}athematical {S}tatistics and {P}robability ({U}niv. {C}alifornia,
  {B}erkeley, {C}alif., 1970/1971), {V}ol. {III}: {P}robability theory}, pages
  333--359, 1972.

\bibitem{UEDA1992166}
U.~Tetsuji and L.~K. William.
\newblock Dynamics of optical pulses in randomly birefringent fibers.
\newblock {\em Physica D: Nonlinear Phenomena}, 55(1):166--181, 1992.

\bibitem{Vil09}
C.~Villani.
\newblock {\em Optimal transport}, volume 338 of {\em Grundlehren der
  Mathematischen Wissenschaften [Fundamental Principles of Mathematical
  Sciences]}.
\newblock Springer-Verlag, Berlin, 2009.
\newblock Old and new.

\bibitem{MR2574753}
L.~Wang, J.~Hong, R.~Scherer, and F.~Bai.
\newblock Dynamics and variational integrators of stochastic {H}amiltonian
  systems.
\newblock {\em Int. J. Numer. Anal. Model.}, 6(4):586--602, 2009.

\bibitem{MR3712946}
X.~Wang, K.~Lu, and B.~Wang.
\newblock Wong-{Z}akai approximations and attractors for stochastic
  reaction-diffusion equations on unbounded domains.
\newblock {\em J. Differential Equations}, 264(1):378--424, 2018.

\bibitem{MR195142}
E.~Wong and M.~Zakai.
\newblock On the convergence of ordinary integrals to stochastic integrals.
\newblock {\em Ann. Math. Statist.}, 36:1560--1564, 1965.

\bibitem{wu2019viscosity}
C.~Wu and J.~Zhang.
\newblock Viscosity solutions to parabolic master equations and
  {M}ckean-{V}lasov sdes with closed-loop controls, 2019.

\end{thebibliography}

\section{Appendix}

\textbf{Proof of Lemma \ref{rd-won}}

\begin{proof}
The local existence of \eqref{lim-sode} and \eqref{inhs} is ensured thanks to the local Lipschitz condition of $f$ and $\sigma$. To obtain a global solution, a priori bound on $H_0(x,p)$ is needed. Denote the solutions of \eqref{inhs} and \eqref{lim-sode} with same initial condition $(x_0,p_0)$ by $(x_t^{\delta},p_t^{\delta}), \delta>0$ and $x_t^{0},p_t^{0}$, respectively. 
Applying the chain rule to $H_0(x_t^{\delta},p_t^{\delta})$ for \eqref{lim-sode} and \eqref{inhs}, we get that 
\begin{align*}
&H_0(x_t^{\delta},p_t^{\delta})=H_0(x_0,p_0)+\int_{0}^t \eta \nabla_p H_0(x_s^{\delta},p_s^{\delta})  \cdot \nabla_x \sigma(x_s) \dot \xi_{\delta}(s)ds\\
&H_0(x_t,p_t)=H_0(x_0,p_0)+\int_0^{\tau} \eta \nabla_p H_0(x_s,p_s) \cdot \nabla_x \sigma(x_s) dB_s\\
&\qquad \qquad \quad +\frac 12\int_0^\tau \eta^2 \nabla_{pp} H_0(x_s,p_s)\cdot (\nabla_x \sigma(x_s),\nabla \sigma(x_s))ds.
\end{align*}
By applying growth condition \eqref{grow-con} and taking expectation on the second equation, we derive that 
\begin{align*}
H_0(x_t^{\delta},p_t^{\delta})
&\le (H_0(x_0,p_0)+\eta C_1T)\exp(\int_0^tc_1\eta|\dot \xi_{\delta}(s)|ds),\\
\E\Big[H_0(x_t,p_t)\Big]
&\le (\E\Big[H_0(x_0,p_0)\Big]+\frac {\eta^2}2 C_1 T) \exp(\int_0^{\tau}c_1 \frac {\eta^2}2 ds).
\end{align*}
The first inequality leads to $H_0(x_t^{\delta},p_t^{\delta})<\infty$ since $\dot \xi_{\delta}(s)=\frac {B_{t_{k+1}}-B_{t_k}}{\delta},$ if $s\in [t_k,t_{k+1}].$
Furthermore, taking expectation on the first inequality, applying Fernique's theorem (see, e.g. \cite{ferniq}) for Gaussian variable and independent increments of $B_t$,
 we get  that
\begin{align*}
\E\Big[H_0(x_t^{\delta},p_t^{\delta})\Big]
&\le C(T,\eta,c_1)(2^{[\frac {t}{\delta}]}(\E\Big[H_0(x_0,p_0)\Big]+1),
\end{align*}
where $[w]$ is the integer part of the real number $w.$
The second inequality yield that $H_0(x_t,p_t)<\infty, a.s,$ and the global existence of the strong solution of \eqref{lim-sode}. Similarly, for $p\ge2$, we have that
\begin{align*}
\E\Big[H_0^p(x_t^{\delta},p_t^{\delta})\Big]
&\le C(T,\eta,c_1,C_1, p)2^{p[\frac {t}{\delta}]}(\E\Big[H_0^p(x_0,p_0)\Big]+1),\\
\E\Big[H_0^p(x_t,p_t)\Big]&\le C(T,\eta,c_1,p)(\E\Big[H_0^p(x_0,p_0)\Big]+1). 
\end{align*}
Furthermore, applying the above bounded moment estimate, we obtain that for $s\le t$,
\begin{align*}
\E\Big[|x(t)-x(s)|^{2p}+|p(t)-p(s)|^{2p}\Big]
&\le C(T,\eta,c_1,C_1,c_0,C_1,p,x_0,p_0) |t-s|^p\\
\E\Big[|x^{\delta}(t)-x^{\delta}(s)|^{2p}+|p(t)-p(s)|^{2p}\Big]
&\le C(T,\eta,c_1,C_1,c_0,C_1,p,x_0,p_0) 2^{[\frac {t}{\delta}]} |t-s|^p.
\end{align*}
However, the above estimate of $x^{\delta}$ is too rough and exponentially depending on $\frac 1{\delta}$. As a consequence, we can not expect any  convergence result. A delicate estimate of $(x^{\delta},p^{\delta})$ is needed.

Assume that $t\in [t_k,t_{k+1}],$ $t_k=k\delta.$ 
Then by using the expansion of \eqref{inhs}, we have that 
\begin{align*}
H_0(x_t^{\delta},p_t^{\delta})
&=H_0(x_0,p_0)-\sum_{j=0}^{k-1}\int_{t_j}^{t_{j+1}}\eta \nabla_p H_0(x_s^{\delta},p_s^{\delta})\cdot \nabla_x \sigma(x_s^{\delta}) d\xi_{\delta}(s)~  \\
&-\int_{t_k}^t \eta \nabla_p H_0(x_s^{\delta},p_s^{\delta})\cdot \nabla_x \sigma(x^{\delta}_s)  d\xi_{\delta}(s)\\
&=H_0(x_0,p_0)-\sum_{j=0}^{k-1}\int_{t_j}^{t_{j+1}}\eta \nabla_p H_0(x_{t_j}^{\delta},p_{t_j}^{\delta})\cdot \nabla_x \sigma(x^{\delta}_{t_j}) d\xi_{\delta}(s)\\
&-\int_{t_k}^t \eta \nabla_p H_0(x_{t_k}^{\delta},p_{t_k}^{\delta})\cdot \nabla_x \sigma(x^{\delta}_{t_k})  d\xi_{\delta}(s)\\
&-\sum_{j=0}^{k-1}\int_{t_j}^{t_{j+1}}\eta \Big(\int_{t_j}^s \nabla_{pp} H_0(x_{r}^{\delta},p_{r}^{\delta}) \cdot (\nabla_x \sigma(x^{\delta}_{r}),-\eta\nabla_x \sigma(x^{\delta}_r) \dot \xi_{\delta}(r)) dr \dot \xi_{\delta}(s) \\
&\quad + \int_{t_j}^s\nabla_{pp} H_0(x_{r}^{\delta},p_{r}^{\delta}) \cdot ( \nabla_x \sigma(x^{\delta}_{r}), -\frac 12(p_r^{\delta})^{\top }d_xg^{-1}(x)p_r^{\delta}-\nabla_x f(x^{\delta}_s)) dr \dot \xi_{\delta}(s)\\
&\quad+ \int_{t_j}^s\nabla_{p} H_0(x_{r}^{\delta},p_{r}^{\delta}) \cdot \nabla_{xx} \sigma(x^{\delta}_{r})g^{-1}(x^{\delta}_r)p^{\delta}_r dr \dot \xi_{\delta}(s)\\
&\quad+\int_{t_j}^s \nabla_{px} H_0(x_r^{\delta},p^{\delta})\cdot(\nabla_x \sigma(x_r^{\delta})\dot \xi_{\delta}(s), g^{-1}(x_r^{\delta})p^{\delta}_r)dr\Big)ds \\
&-\int_{t_k}^t\eta \Big(\int_{t_k}^s \nabla_{pp} H_0(x_{r}^{\delta},p_{r}^{\delta}) \cdot (\nabla_x \sigma(x^{\delta}_{r}),-\eta \nabla_x \sigma(x^{\delta}_r) \dot \xi_{\delta}(r)) dr \dot \xi_{\delta}(s) \\
&\quad + \int_{t_k}^s\nabla_{pp} H_0(x_{r}^{\delta},p_{r}^{\delta}) \cdot (\nabla_x \sigma(x^{\delta}_{r}), -\frac 12(p_r^{\delta})^{\top}d_xg^{-1}(x^{\delta}_r)p^{\delta}_r-\nabla_x f(x^{\delta}_s)) dr \dot \xi_{\delta}(s)\\
&\quad+ \int_{t_k}^s\nabla_{p} H_0(x_{r}^{\delta},p_{r}^{\delta}) \cdot \nabla_{xx} \sigma(x^{\delta}_{r})g^{-1}(x^{\delta_r})p^{\delta}_r dr \dot \xi_{\delta}(s)\\
&\quad + \int_{t_k}^{s} \nabla_{px} H_0(x_r^{\delta},p^{\delta})\cdot(\nabla_x \sigma(x_r^{\delta})\dot \xi_{\delta}(s)),g^{-1}(x_r^{\delta})p_r^{\delta})dr
\Big)ds\\
&=: H_0(x_0,p_0) +\sum_{j=0}^{k-1}I_{j}^1+I^1_k(t) \\
&+\sum_{j=0}^{k-1} (I_{j}^{21}+I_{j}^{22}+I_j^{23}+I_j^{24})+I_{k}^{21}(t)+I_{k}^{22}(t)+I_k^{23}(t)+I_k^{24}(t).
\end{align*}
 Making use of the growth condition \eqref{grow-con}, we have that 
\begin{align*}
&\sum_{j=0}^{k-1} (I_{j}^{21}+I_{j}^{22}+I_j^{23}+I_j^{24})+I_{k}^{21}(t)+I_{k}^{22}(t)+I_k^{23}(t)+I_k^{24}(t)\\
&\le \sum_{j=0}^{k-1}\int_{t_j}^{t_{j+1}}(C_1+c_1 H_0(x_s^{\delta},p_s^{\delta}))|\dot \xi_{\delta}(s)|^2\delta ds 
+\sum_{j=0}^{k-1} \int_{t_j}^{t_{j+1}} (C_1+c_1 H_0(x_s^{\delta},p_s^{\delta}))|\dot \xi_{\delta}(s)| \delta ds\\
&+\int_{t_k}^{t}(C_1+c_1 H_0(x_s^{\delta},p_s^{\delta}))|\dot \xi_{\delta}(s)|^2\delta ds 
+ \int_{t_k}^{t} (C_1+c_1 H_0(x_s^{\delta},p_s^{\delta}))|\dot \xi_{\delta}(s)| \delta ds\\
&=\int_{0}^{t}(C_1+c_1 H_0(x_s^{\delta},p_s^{\delta})) |\dot \xi_{\delta}(s)|^2\delta ds 
+ \int_{0}^{t} (C_1+c_1 H_0(x_s^{\delta},p_s^{\delta}))|\dot \xi_{\delta}(s)| \delta ds. 
\end{align*}
By using the Gronwall--Bellman inequality, we obtain that 
\begin{align*}
 H_0(x_t^{\delta},p_t^{\delta})&\le \exp(\int_0^t c_1(|\dot \xi_{\delta}(s)|^2 +|\dot \xi_{\delta}(s)|)\delta ds)(H_0(x_0,p_0)+CT+|\sum_{j=0}^{k-1}I_{j}^1+I^1_k(t)|).
\end{align*}
For simplicity, assume that $T=K\delta.$ Denote $[t]_{\delta}=t_k=k\delta$ if $t\in [t_k,t_{k+1}).$ The definition of $\xi_{\delta}(s)$ yields that $s\in [t_j,t_{j+1}]$
\begin{align*}
|\dot \xi_{\delta}(s)|^2\delta+|\dot \xi_{\delta}(s)|\delta =|\frac {B(t_{j+1})-B(t_j)}\delta|^2\delta +|B(t_{j+1})-B(t_j)|.
\end{align*}


Define a stopping time $\tau_R=\inf \{t\in [0,T]| \int_0^{[t]_{\delta}}|\dot \xi_{\delta}|^2 \delta ds \ge R\}.$ 
The stopping time is well-defined since the quadratic variation process of Brownian motion is bounded in $[0,T].$
Then taking $t\le \tau_{R}$ and using H\"older's inequality, then it holds that 
\begin{align}\label{con-gron}
    H_0(x_t^{\delta},p_t^{\delta})&\le \exp(\int_{[t]}^t c_1(|\dot \xi_{\delta}(s)|^2+|\dot \xi_{\delta}| ds)\exp(C(R+T))(H_0(x_0,p_0)+CT+|\sum_{j=0}^{k-1}I_{j}^1+I^1_k(t)|)\\\nonumber
    &\le \exp(\int_{[t]}^t c_1(\frac 32|\dot \xi_{\delta}(s)|^2) ds)\exp(C(R+T)) H_0(x_0,p_0)\\\nonumber
    &+ \exp(C(R+T)) \exp(\int_{[t]}^t c_1\frac 32|\dot \xi_{\delta}(s)|^2 ds)
    \left|\int_0^{[t]}-\eta \nabla_p H_0(x_{[s]_\delta}^{\delta},p_{[s]_{\delta}}^{\delta})\cdot \nabla_x \sigma(x_{[s]_{\delta}}) dB(s)\right|\nonumber\\\nonumber
    & + \exp(C(R+T)) \exp(\int_{[t]}^t (c_1\frac 32|\dot \xi_{\delta}(s)|^2 ds)
    \left|\int_{[t]}^t-\eta \nabla_p H_0(x_{[s]_\delta}^{\delta},p_{[s]_{\delta}}^{\delta})\cdot \nabla_x \sigma(x_{[s]_{\delta}}) \dot \xi_{\delta}(s) ds\right|. 
\end{align}

Similarly, one could obtain a analogous estimate of \eqref{con-gron}  with the integral over $[t_{k-1},t_{k}]$, where $t_k$, $k\le K$,  $t_K\le \tau_R.$ 
By the Cauchy inequality abd taking expectation on both sides of \eqref{con-gron},   applying the Burkholder--Davis--Gundy inequality (see e.g, \cite{burkholder1972integral}) and using the independent increments of Brownian motion, 
we get 
\begin{align*}
 &\E[H_0^2(x_{t_k}^{\delta},p_{t_k}^{\delta})] \\
 &\le 3\E\Big[\exp(\int_{t_{k-1}}^{t_k} (3c_1|\dot \xi_{\delta}(s)|^2 ds)\Big]\exp(2C(R+T)) \E\Big[H_0^2(x_0,p_0)\Big]\\
 &+ 3\exp(2C(R+T)) \E [\exp(\int_{t_{k-1}}^{t_k} 3c_1|\dot \xi_{\delta}(s)|^2 ds)]
    \E\Big[\Big|\int_0^{t_{k-1}}-\eta \nabla_p H_0(x_{[s]_\delta}^{\delta},p_{[s]_{\delta}}^{\delta})\cdot \nabla_x \sigma(x_{[s]_{\delta}}) dB(s)\Big|^2\Big]\\ 
&+  3\exp(2C(R+T)) \E\Big[\exp(\int_{t_{k-1}}^{t_k} 3c_1|\dot \xi_{\delta}(s)|^2 ds)  |B(t_{k})-B(t_{k-1})|^2 \\
&\times \big|\eta \nabla_p H_0(x_{t_{k-1}}^{\delta},p_{t_{k-1}}^{\delta})\cdot \nabla_x \sigma(x_{t_{k-1}}) \big|^2\Big]\\
&\le 3\E\Big[\exp(\int_{t_{k-1}}^{t_k} (3c_1|\dot \xi_{\delta}(s)|^2 ds)\Big]\exp(2C(R+T)) \E\Big[H_0^2(x_0,p_0)\Big]\\
 &+ 3\exp(2C(R+T)) \E [\exp(\int_{t_{k-1}}^{t_k} 3c_1|\dot \xi_{\delta}(s)|^2 ds)]
    \E\Big[\int_0^{t_{k-1}}(C_1+c_1 H_0(x_{[s]_{\delta}}^{\delta},p_{[s]_{\delta}}^{\delta}))^2 ds\Big]\\
&+  3\exp(2C(R+T)) \E\Big[\exp(\int_{t_{k-1}}^{t_k} 3c_1|\dot \xi_{\delta}(s)|^2 ds)  |B(t_{k})-B(t_{k-1})|^2\Big] \\
&\times \E\Big[(C_1+c_1 H_0^2(x_{t_{k-1}}^{\delta},p_{t_{k-1}}^{\delta}))\Big].
\end{align*}
 

Applying the Fernique theorem 
and choosing sufficient small $\delta$ such that $12c_1\delta<1,$
then we have that 
\begin{align*}
&\E [\exp(\int_{t_{k-1}}^{t_k} 3c_1|\dot \xi_{\delta}(s)|^2 ds)]\le C,\\
&\E\Big[\exp(\int_{t_{k-1}}^{t_k} 3c_1|\dot \xi_{\delta}(s)|^2 ds)  |B(t_{k})-B(t_{k-1})|^2\Big]\\
&\le  \sqrt{\E [\exp(\int_{t_{k-1}}^{t_k} 6c_1|\dot \xi_{\delta}(s)|^2 ds)]}\sqrt{\E \Big[|B(t_{k})-B(t_{k-1})|^4\Big]} \le C\delta.
\end{align*}
The above estimation gives 
\begin{align*}
    \E[H_0^2(x_{t_k}^{\delta},p_{t_k}^{\delta})]
    &\le 3 \exp(2C(R+T))C \E [H_0^2(x_0,p_0)]\\
    &+6 \exp(2C(R+T))C\int_{0}^{t_{k-1}}\E \Big[(C_1^2+c_1^2H_0^2(x_{[s]_{\delta}}^{\delta},p_{[s]_{\delta}}^{\delta}))\Big]ds\\
    &+6 \exp(2C(R+T)C\delta \E\Big[C_1^2+c_1^2H_0^2(x_{t_{k-1}}^{\delta},p_{t_{k-1}}^{\delta})\Big].
\end{align*}
Then the Grownall's inequality yield that 
\begin{align*}
\E[H_0^2(x_{t_k}^{\delta},p_{t_k}^{\delta})]
&\le \exp(6TCc_1^2\exp(2C(R+T))) \\
&\times \Big(3\exp(2C(R+T))C\E [H_0^2(x_0,p_0)] + 6C_1^2T C \exp(2C(R+T))\Big)
\end{align*}
Combining the above estimates with \eqref{con-gron} and  the Burkholder--Davis--Gundy inequality, we conclude that 
\begin{align*}
\sup_{t\in [0,\tau^R)}\E[H_0^2(x_{t}^{\delta},p_{t}^{\delta})] 
&\le (\exp(6TCc_1^2\exp(2C(R+T)))+C) \\
&\times \Big(3\exp(2C(R+T))C\E [H_0^2(x_0,p_0)] + 6C_1^2T C \exp(2C(R+T))\Big)\\
& =:C_R.
\end{align*}
Similarly, by choosing sufficient small $\delta$, we have that for $t\in [0,\tau^R)$,
\begin{align*}
\E[H_0^{p}(x_{t}^{\delta},p_{t}^{\delta})] 
\le C_{R,p}<\infty.
\end{align*}
As a consequence, by again using \eqref{con-gron}, we obtain that  
\begin{align*}
    \E\Big[\sup_{t\in [0,\tau^R)}H_0^p(x_{t}^{\delta},p_{t}^{\delta})\Big]
    &\le C_{R,p}<\infty.
\end{align*}


Next we show the convergence in probability of the solution of \eqref{inhs} to that of \eqref{lim-sode}.
Introduce another stopping time $\tau_{R_1}:=\inf\{t\in [0,T]| |x_t|+|p_t|\ge R_1, |x^{\delta}_{[t]_{\delta}}|+|p^{\delta}_{[t]_{\delta}}|\ge R_1 \}.$ Let $t\in [0,\tau_{R}\wedge \tau_{R_1}).$ By using the polynomial growth condition of $f,\sigma$ and the fact that $\sigma$ is independent of $p$, we obtain that
\begin{align*}
&|x^{\delta}(t)-x(t)|^2\\
&= |x^\delta(0)-x(0)|^2 +\int_{0}^t2\<x^{\delta}(s)-x(s),g^{-1}(x^{\delta}(s)) p^{\delta}(s)-g^{-1}(x(s))p(s)\>ds\\
&\le |x^{\delta}(0)-x(0)|^2+\int_0^t C_g(1+|p(s)|)(|x^{\delta}(s)-x(s)|^2+|p^{\delta}(s)-p(s)|^2) ds,\\
&|p^{\delta}(t)-p(t)|^2\\
&=
\int_0^t \<- (p^{\delta}(s))^{\top} d_x g^{-1}(x^\delta(s)) p^{\delta}(s) +  p(s)^{\top} d_x g^{-1}(x(s)) p(s) ,p^{\delta}(s)-p(s)\>ds
\\
&+\int_0^t2\<-\nabla_x f(x^{\delta}(s))+\nabla_x f(x(s)),p^{\delta}(s)-p(s)\>ds\\
&-\int_0^t2\eta \<p^{\delta}(s)-p(s), \nabla_x \sigma(x^{\delta}(s)) d\xi_{\delta}(s)-\nabla_x\sigma(x(s))dB_t\>\\
&\le C_g \int_0^t (1+|x^{\delta}(s)|)(|p^{\delta}(s)|^2+|p(s)|^2)(|p^{\delta}(s)-p(s)|^2+|x^{\delta}(s)-x(s)|^2) ds\\
&+ C_f \int_0^t (1+|x(s)|^{p_f}+|x^{\delta}|^{p_f})(|p^{\delta}(s)-p(s)|^2+|x^{\delta}(s)-x(s)|^2)ds\\
&-\int_0^t2\eta \<p^{\delta}(s)-p(s), \nabla_x \sigma(x^{\delta}(s)) d\xi_{\delta}(s)-\nabla_x\sigma(x(s))dB_s\>,
\end{align*}
where $C_g$ and  $C_f$ are constants depending on $f$ and $g$.
To deal with the last term, we split it as follows,
\begin{align*}
&\int_0^t2\eta \<p^{\delta}(s)-p(s), \nabla_x \sigma(x^{\delta}(s)) d\xi_{\delta}(s)-\nabla_x\sigma(x(s))dB_s\>\\
&=2\eta\int_0^t \<p^{\delta}([s]_{\delta})-p([s]_{\delta}), \nabla_x \sigma(x^{\delta}(s)) d\xi_{\delta}(s)-\nabla_x\sigma(x(s))dB_s\>\\
&+2\eta\int_0^t \<p^{\delta}(s)-p(s)-p^{\delta}([s]_{\delta})+p([s]_{\delta}), \nabla_x \sigma(x^{\delta}(s)) d\xi_{\delta}(s)-\nabla_x\sigma(x(s))dB_s\>\\
&=2\eta\int_0^t \<p^{\delta}([s]_{\delta})-p([s]_{\delta}), \nabla_x \sigma(x^{\delta}([s]_{\delta})) d\xi_{\delta}([s]_{\delta})-\nabla_x\sigma(x([s]_{\delta}))dB_s\> \\ 
&+2\eta\int_0^t \<p^{\delta}([s]_{\delta})-p([s]_{\delta}), (\nabla_x \sigma(x^{\delta}(s))- \nabla_x \sigma(x^{\delta}([s]_{\delta})))d\xi_{\delta}(s)-(\nabla_x\sigma(x(s))-\nabla_x\sigma(x([s]_{\delta})))dB_s\>\\
&+ 2\eta\int_0^t \<p^{\delta}(s)-p(s)-p^{\delta}([s]_{\delta})+p([s]_{\delta}), \nabla_x \sigma(x^{\delta}([s]_{\delta})) d\xi_{\delta}(s)-\nabla_x\sigma(x([s]_{\delta})dB_s\>\\
&+ 2\eta\int_0^t \<p^{\delta}(s)-p(s)-p^{\delta}([s]_{\delta})+p([s]_{\delta}), (\nabla_x \sigma(x^{\delta}(s)-\nabla_x \sigma(x^{\delta}([s]_{\delta})) ) d\xi_{\delta}(s)\\
&\quad -(\nabla_x\sigma(x(s)-\nabla_x\sigma(x([s]_{\delta}))dB_s\>\\
& =: II^1+II^2+II^3+II^4.
\end{align*}
Taking expectation on $II^1$, using the property of the discrete martingale,  the a prior estimates for $H_0(x_t,p_t)$ and $H_0(x_t^{\delta},p_t^{\delta})$ and H\"older's inequality,
we have that 
\begin{align*}
    \E[II^1]&=0,\\
\E[II^2]
&\le 2\eta\int_0^t \E \Big[\<p^{\delta}([s]_{\delta})-p([s]_{\delta}), \int_{[s]_{\delta}}^{s} (\nabla_{xx} \sigma(x^{\delta}(r))\cdot (g^{-1}(x^{\delta}(r))p^{\delta}(r)) dr d\xi_{\delta}(s)\>\Big]\\
& -2\eta\int_0^t \E \Big[\<p^{\delta}([s]_{\delta})-p([s]_{\delta}),\int_{[s]_{\delta}}^{s} (\nabla_{xx} \sigma(x(r))\cdot(g^{-1}(x^{\delta}(r)p^{\delta}(r)) dr dB_s\>\Big]\\
&\le C(R_1) \delta^{\frac 12}. 
\end{align*}
Similar arguments lead to $\E[II^4]\le C(R_1)\delta^{\frac 12}.$
For the term $II^3$, applying the continuity estimate of $x_t$ and $x^{\delta}_t,$ as well as independent increments of the Brownian motion, we get 
\begin{align*}
&\E[II^3]\\
&\le C(R_1)\delta^{\frac 12} +2\eta^2\E\Big[\int_0^{[t]_{\delta}} \<\int_{[s]_{\delta}}^s \nabla_x \sigma(x_{[r]_{\delta}}^{\delta}) d\xi_{\delta}(r) -\int_{[s]_{\delta}}^s \nabla_x \sigma(x_{[r]_{\delta}}) dB_r, \\
&\quad \nabla_x \sigma(x^{\delta}([s]_{\delta})) d\xi_{\delta}(s)-\nabla_x\sigma(x([s]_{\delta})dB_s\>\Big]\\
&\le C(R_1)\delta^{\frac 12}
+2\eta^2\E\Big[\int_0^{[t]_{\delta}} |\nabla_x \sigma(x^{\delta}_{[s]_{\delta}})|^2\frac{s-[s]_{\delta}}{\delta^2}(B([s]_{\delta}+\delta)-B([s]_{\delta}))^2ds\Big]\\
&-2\eta^2\E \Big[\int_0^{[t]_{\delta}} \<\nabla_x \sigma(x^{\delta}_{[s]_{\delta}}), \nabla_x \sigma(x_{[s]_{\delta}})\> \frac{s-[s]_{\delta}}{\delta^2}(B([s]_{\delta}+\delta)-B([s]_{\delta}))^2ds \Big]\\
&-2\eta^2 \E \Big[\int_0^{[t]_{\delta}} \<\nabla_x \sigma(x^{\delta}_{[s]_{\delta}}), \nabla_x \sigma(x_{[s]_{\delta}})\> \frac{ B([s]_{\delta}+\delta)-B([s]_{\delta})}{\delta} (B(s)-B([s]_{\delta}))ds\Big]\\
&+2\eta^2 \E \Big[\int_0^{[t]_{\delta}} \<\nabla_x \sigma(x_{[s]_{\delta}}), \nabla_x \sigma(x_{[s]_{\delta}})\> \frac{ B([s]_{\delta}+\delta)-B([s]_{\delta})}{\delta} (B(s)-B([s]_{\delta}))ds\Big]\\
&\le  C(R_1)\delta^{\frac 12}+2\eta^2 \int_0^{[t]_{\delta}}\E \Big[|\nabla_x \sigma(x_{[s]_{\delta}}^{\delta})-\nabla_x \sigma(x_{[s]_{\delta}})|^2\Big]ds\\
&\le C(R_1)\delta^{\frac 12}+ \int_0^{t}C(R_1)\E \Big[| x_s^{\delta}-x_s|^2\Big]ds,
\end{align*}
where $C(R_1)>0$ is monotone with $R_1.$
Combining the above estimates, we achieve that 
\begin{align*}
\E[|x^{\delta}(t)-x(t)|^2]&\le 
\int_0^t C_g(1+C_{R_1}) (\E[|x^{\delta}(s)-x(s)|^2]+ \E[|p^{\delta}(s)-p(s)|^2])ds \\
\E[|p^{\delta}(t)-p(t)|^2]
&\le \int_0^t (C_g+C_f)(1+C_{R_1})(\E[|x^{\delta}(s)-x(s)|^2]+ \E[|p^{\delta}(s)-p(s)|^2])ds
+C(R_1)\delta^{\frac 12}.
\end{align*}
Then the Gronwall's inequality implies that 
\begin{align}\label{rough-strong}
\E[|x^{\delta}(t)-x(t)|^2]+\E[|p^{\delta}(t)-p(t)|^2]
&\le \exp(2(C_g+C_f)(1+C_{R_1})T)C(R_1)\delta^{\frac 12}.
\end{align}
By making use of \eqref{rough-strong} and Chebshev's inequality, we conclude that 
\begin{align*}
&\mathbb P(|x^{\delta}(t)-x(t)|+|p^{\delta}(t)-x(t)|\ge \epsilon)\\
&\le   \mathbb P(\{|x^{\delta}(t)-x(t)|+|p^{\delta}(t)-x(t)|\ge \epsilon\}\cap \{t<\tau_R\}\cap \{t<\tau_{R_1}\})  \\
&+ \mathbb P(\{|x^{\delta}(t)-x(t)|+|p^{\delta}(t)-x(t)|\ge \epsilon\}\cap \{t\ge \tau_R\})\\
&+\mathbb P(\{|x^{\delta}(t)-x(t)|+|p^{\delta}(t)-x(t)|\ge \epsilon\}\cap \{t< \tau_R\}\cap \{t\ge \tau_{R_1}\})\\
&\le 2\frac {\E\Big[|x^{\delta}(t)-x(t)|^2+|p^{\delta}(t)-x(t)|^2\Big]}{\epsilon^2} \\
&+ \frac {\E\Big[\int_0^{t}|\dot \xi_{\delta}(s)|^2\delta ds\Big]}{R}+\frac {\E \Big[|x(t)|+|p(t)|+|x^{\delta}(t)|+|p^{\delta}(t)|\Big]}{R_1}\\
&\le \frac 2{\epsilon^2}\exp(2(C_g+C_f)(1+C_{R_1})T)C(R_1)\delta^{\frac 12}+ \frac C{R}+C\frac {1+C_R}{R_1}.
\end{align*}
Here, $\E [|x(t)|+|p(t)|+|x^\delta(t)|+|p^\delta(t)|]<C(1+C_R)$ is ensured by $\E[\sup\limits_{t\in [0,\tau^R)} H_0^{2}(x_{t}^{\delta},p_{t}^{\delta})] 
\le C_{R}.$
Taking limit on $\delta \to 0,$ $R_1\to \infty$, and $R \to \infty$ leads to 
\begin{align*}
   \lim_{\delta\to 0}\mathbb P(|x^{\delta}(t)-x(t)|+|p^{\delta}(t)-p(t)|> \epsilon)=0.
\end{align*}
Similarly, one could utilize the properties of martingale and obtain the following estimate, for large enough $q>0,$ 
\begin{align*}
\E[|x^{\delta}(t)-x(t)|^{q}]+\E[|p^{\delta}(t)-p(t)|^q]
&\le C_q \exp(C_q(C_g+C_f)(1+C_{R_1})T)C(R_1)\delta^{\frac q 2 -1}.
\end{align*}
This implies that for large enough $q>4$, 
\begin{align*}
&\E[\sup_{k\le K}\sup_{t\in[t_{k-1},t_{k}]}|x^{\delta}(t)-x(t)|^{q}]+\E[\sup_{k\le K}\sup_{t\in[t_{k-1},t_{k}]}|p^{\delta}(t)-p(t)|^q]\\
&\le \sum_{k=0}^{K-1}\E[\sup_{t\in[t_{k-1},t_{k}]}|x^{\delta}(t)-x(t)|^{q}]+\E[\sup_{t\in[t_{k-1},t_{k}]}|p^{\delta}(t)-p(t)|^q]\\
&\le C_q K \exp(C_q(C_g+C_f)(1+C_{R_1})T)C(R_1)\delta^{\frac q 2 -1}\\
&\le C_q  \exp(C_q(C_g+C_f)(1+C_{R_1})T)C(R_1)\delta^{\frac q 2 -2}.
\end{align*}
Combining the above estimate and applying the Chebshev's inequality, we further obtain 
\begin{align*}
 \lim_{\delta\to 0}\mathbb P(\sup_{t\in[0,T]}|x^{\delta}(t)-x(t)|+\sup_{t\in [0,T]}|p^{\delta}(t)-p(t)|> \epsilon)=0.
\end{align*}
\end{proof}

\end{document}